\documentclass{article}
\usepackage[english]{babel}
\usepackage[utf8]{inputenc}
\usepackage{amsfonts}
\usepackage{latexsym}
\usepackage{amsmath}
\usepackage{amsthm}
\usepackage{amssymb}
\usepackage{mathabx}
\usepackage{bm}
\usepackage{enumerate}
\usepackage{csquotes}
\usepackage{hyperref}
\usepackage{cleveref}
\usepackage{comment}
\usepackage[backend=biber,
            style=alphabetic,
            giveninits=true,
            isbn=false,
            doi=false,
            url=false,
            sorting=nyvt,
            maxbibnames=4]{biblatex}
\renewbibmacro{in:}{}
\hypersetup{breaklinks=true}

\newtheorem{theorem}{Theorem}
\newtheorem{lemma}[theorem]{Lemma}

\newtheorem{proposition}[theorem]{Proposition}
\newtheorem{lettertheorem}{Theorem}

\theoremstyle{definition}

\theoremstyle{remark}
\newtheorem{remark}[theorem]{Remark}

\numberwithin{equation}{section}

\DeclareMathOperator*{\esssup}{ess\,sup}

\newcommand{\norm}[2]{\left\|{#1}\right\|_{{#2}}}
\newcommand{\mathdm}{\,\mathrm{dm}}
\newcommand{\cB}{\mathcal{B}}
\newcommand{\cC}{\mathcal{C}}
\newcommand{\cD}{\mathcal{D}}
\newcommand{\cE}{\mathcal{E}}
\newcommand{\cF}{\mathcal{F}}

\newcommand{\cH}{\mathcal{H}}
\newcommand{\cK}{\mathcal{K}}

\newcommand{\cU}{\mathcal{U}}
\newcommand{\cV}{\mathcal{V}}
\newcommand{\cW}{\mathcal{W}}
\newcommand{\bfD}{\bm{\mathcal{D}}}
\newcommand{\e}{\varepsilon}

\newcommand{\integer}{\mathbb{Z}}
\newcommand{\reals}{\mathbb{R}}
\newcommand{\complex}{\mathbb{C}}
\newcommand{\convex}{\cK_{\mathrm{bcs}}}

\newcommand{\lp}[1]{\ell^{#1}}
\newcommand{\Lp}[1]{\mathrm{L}^{#1}}
\newcommand{\1}{\mathbf{1}}
\newcommand{\var}{\,\cdot\,}
\newcommand{\tr}{\mathrm{tr}}
\newcommand{\BMOprodD}{\ensuremath\text{BMO}_{\text{prod},\bfD}}
\newcommand{\matrices}[2]{\mathrm{M}_{{#1}}({#2})}

\newcommand{\complexball}[1]{\mathbf{\overline{B}}_{\complex^{{#1}}}}
\newcommand{\complexsphere}[1]{\mathbf{{S}}_{\complex^{{#1}}}}

\newcommand{\Addresses}{{
  \medskip
  \footnotesize

  \noindent Spyridon~Kakaroumpas (ORCID: \href{https://orcid.org/0000-0003-3676-3096}{0000-0003-3676-3096})\\
  \textsc{W\"{u}rzburg Mathematics Center of Communication and Interaction,\\
          Julius-Maximilians-Universit\"{a}t W\"{u}rzburg,\\
          Campus Hubland Nord,\\
          Emil-Fischer-Stra{\ss}e 41,\\
          97074 W\"{u}rzburg, Germany}\\
  \textit{E-mail address: } \texttt{spyridon.kakaroumpas@uni-wuerzburg.de}

  \medskip

  \noindent Odí~Soler~i~Gibert (ORCID: \href{https://orcid.org/0000-0003-2746-9565}{0000-0003-2746-9565})\\
  \textsc{Universitat Politècnica de Catalunya - BarcelonaTech (UPC),\\
          Edifici C3, despatx C3 206,\\
          Carrer de Jordi Girona 1--3,\\
          08034 Barcelona, Catalunya}\\
  \textit{E-mail address: } \texttt{odi.soler@upc.edu}

}}

\begin{document}

    \title{Vector valued estimates for matrix weighted maximal operators and
           product \texorpdfstring{$\mathrm{BMO}$}{BMO}}

    \author{Spyridon Kakaroumpas
            and Od{\'i} Soler i Gibert\footnote{Corresponding author.}
            \thanks{First author is supported by the Alexander von Humboldt-Stiftung.
            Second author is supported by
            the Generalitat de Catalunya (grant 2021 SGR 00071),
            the Spanish Ministerio de Ciencia e Innovaci\'{o}n (project PID2021-123151NB-I00),
            and the Spanish Research Agency (Mar\'{i}a de Maeztu Program CEX2020-001084-M).}}

    \date{}

    \maketitle

    \begin{abstract}
        We consider maximal operators acting on vector valued functions,
        that is, functions taking values on $\complex^d,$
        that incorporate matrix weights in their definitions.
        We show vector valued estimates, in the sense of Fefferman--Stein inequalities,
        for such operators.
        These are proven using an extrapolation result for convex body valued functions
        due to Bownik and Cruz-Uribe.
        Finally, we show an $\mathrm{H}^1$-$\mathrm{BMO}$ duality for matrix valued functions
        and we apply the previous vector valued estimates to show upper bounds
        for biparameter paraproducts.
        For the reader's convenience, we include an appendix explaining how to
        adapt the extrapolation for real convex body valued functions of Bownik and Cruz-Uribe
        to the setting of complex convex body valued functions that we treat.\\
        \textit{2020 Mathematics Subject Classification:} 42B25, 42B30, 42B35.\\
        \textit{Keywords:} convex set valued functions, maximal operators,
                           matrix weights, Fefferman--Stein vector valued estimates,
                           paraproducts, matrix weighted matrix valued $\mathrm{BMO},$
                           $\mathrm{H}^1$-$\mathrm{BMO}$ duality.
    \end{abstract}

    \section{Introduction}

    This paper deals with matrix weighted extensions of the Fefferman--Stein vector valued inequalities for the Hardy--Littlewood maximal function and the closely connected topic of matrix weighted extensions of the product BMO spaces. Before stating our main results, we briefly recall some historical background on each area.

    The classical Fefferman--Stein vector valued inequalities, first proved by C.~Fefferman and E.~Stein \cite{Fefferman_Stein_1971}, state that for all $1<p,q<\infty$ one has
    \begin{equation}
        \label{classical Fefferman Stein}
        \left\Vert\left(\sum_{k=1}^{\infty}|Mf_k|^{q}\right)^{1/q}\right\Vert_{\Lp{p}(\reals^n)}\leq C(n,p,q)\left\Vert\left(\sum_{k=1}^{\infty}|f_k|^{q}\right)^{1/q}\right\Vert_{\Lp{p}(\reals^n)},
    \end{equation}
    for any sequence $\{f_k\}_{k=1}^{\infty}$ of (say) locally integrable complex valued functions on $\reals^n,$ where $M$ is the usual (uncentered) Hardy--Littlewood maximal function. That means
    \begin{equation*}
        Mf(x)\coloneq \sup_{Q\ni x}\frac{1}{|Q|}\int_{Q}|f(x)|\mathdm(x),\quad x\in\reals^n,
    \end{equation*}
    where the supremum ranges over all cubes $Q$ in $\reals^n$ (with faces parallel to the coordinate hyperplanes), $|E|$ is the Lebesgue measure of a Lebesgue measurable set $E\subseteq\reals^n,$ and $\mathdm(x)$ denotes integration with respect to Lebesgue measure. The constant $C(n,p,q)$ in \eqref{classical Fefferman Stein} depends only on $n,$ $p$ and $q.$ In \cite{Fefferman_Stein_1971} also a weak type version of \eqref{classical Fefferman Stein} is proved.

    The estimate \eqref{classical Fefferman Stein} is actually a special case of more general bounds for \emph{vector valued extensions} of operators. That is, given some (not necessarily linear) operator $T$ acting boundedly on $\Lp{p}(\reals^n)$ for some $1<p<\infty,$ we seek to find those $1<q<\infty$ satisfying an estimate of the form
    \begin{equation}
        \label{vector valued inequalities}
        \left\Vert\left(\sum_{k=1}^{\infty}|Tf_k|^{q}\right)^{1/q}\right\Vert_{\Lp{p}(\reals^n)}\leq C(T,n,p,q)\left\Vert\left(\sum_{k=1}^{\infty}|f_k|^{q}\right)^{1/q}\right\Vert_{\Lp{p}(\reals^n)}
    \end{equation}
    for any sequence $\{f_k\}^{\infty}_{k=1}.$ Such inequalities seem to have been studied for the first time systematically by J.-L.~Rubio de Francia \cite{Rubio_de_Francia_1985}. For a thorough modern exposition of the methods in \cite{Rubio_de_Francia_1985} we refer to \cite{Tao_notes_5_247A} as well as \cite{Grafakos_Book_1}.
    
    The ideas in \cite{Rubio_de_Francia_1985} already hinted at an intimate connection between \emph{extrapolation} and bounds for vector valued extensions as in \eqref{vector valued inequalities}. By an extrapolation problem one understands the following. Given an operator $T$ acting on (suitable) functions on $\reals,$ we assume that for some $1<p<\infty$ it is already known or given that for all weights (that means, a.e.~positive locally integrable functions) $w$ on $\reals^n$ that belong to some class $C(p),$ one has the estimate $\Vert Tf\Vert_{\Lp{p}(w)}\leq C(T,n,p,w)\Vert f\Vert_{\Lp{p}(w)}.$ Given this information, find all $1<q<\infty$ as well as associated classes $C(q)$ of weights on $\reals^n,$ such that for any $w\in C(q)$ one has an estimate of the form $\Vert Tf\Vert_{\Lp{q}(w)}\leq C(T,n,q,w)\Vert f\Vert_{\Lp{q}(w)}.$ Rubio de Francia \cite{Rubio_de_Francia_1984} solved completely the extrapolation problem in the case that $C(p)$ coincides with the \emph{Muckenhoupt $A_p$ class}, that is $w\in A_p$ if and only if
    \begin{equation*}
        [w]_{A_p}\coloneq \sup_{Q}\left(\frac{1}{|Q|}\int_{Q}w(x)\mathdm(x)\right)\left(\frac{1}{|Q|}\int_{Q}w(x)^{-1/(p-1)}\mathdm(x)\right)^{p-1}<\infty,
    \end{equation*}
    where the supremum ranges again over all cubes $Q\subseteq\reals^n.$ In this case, beginning with \emph{any} fixed $1<p<\infty$ and an extrapolation hypothesis holding for all weights $w\in A_p,$ the extrapolation problem is solvable for \emph{any} $1<q<\infty$ and all weights $w\in A_q.$ The extrapolation theorem of Rubio de Francia was subsequently further refined by various authors, until a sharp quantitative version of it was proved by O.~Dragi{\v c}evi{\'c}, L.~Grafakos, M.~C.~Pereyra and S.~Petermichl \cite{Dragicevic2005} (see also \cite{Duoandikoetxea_2011}). A very thorough treatment of various forms of extrapolation with extensive historical background can be found in \cite{Book_Extrapolation}. In fact, this method is so powerful, that it naturally yields \emph{weighted estimates} for vector valued extensions, that is
    \begin{equation}
        \label{weighted vector valued inequalities}
        \left\Vert\left(\sum_{k=1}^{\infty}|Tf_k|^{q}\right)^{1/q}\right\Vert_{\Lp{p}(w)}\leq C(T,n,p,q,w)\left\Vert\left(\sum_{k=1}^{\infty}|f_k|^{q}\right)^{1/q}\right\Vert_{\Lp{p}(w)},
    \end{equation}
    as explained in detail in \cite{Book_Extrapolation}.

    Inequalities of the form \eqref{weighted vector valued inequalities} are a major tool for estimating operators arising naturally when decomposing \emph{biparameter operators} or \emph{bicommutators} in simpler, localized pieces. Such a decomposition for the so called Journe\'e operators was established by H.~Martikainen \cite{Martikainen_2012} (generalizing an analogous decomposition proved earlier by T.~Hyt\"{o}nen \cite{Hytonen_2012} for Calder\'{o}n--Zygmound operators in the context of the solution of the $A_2$ problem). I.~Holmes, S.~Petermichl and B.~Wick \cite{Holmes_Petermichl_Wick_2018} showed that these localized pieces can be estimated in terms of a \emph{weighted product BMO} space. In the following, we recall the relevant definitions.
    
    The classical space $\text{BMO}(\reals^n)$ consists of all locally integrable functions $b$ on $\reals^n$ such that
    \begin{equation*}
        \Vert b\Vert_{\text{BMO}(\reals^n)}\coloneq \sup_{Q}\frac{1}{|Q|}\int_{Q}|b(x)-\langle b\rangle_{Q}|\mathdm(x)<\infty,
    \end{equation*}
    where the supremum is taken over all cubes $Q\subseteq\reals^n,$ and we denote $\langle b \rangle_Q\coloneq \frac{1}{|Q|}\int_{Q}b(x)\mathdm(x).$ The importance of this space is two-fold. First, it is the dual space to the (real variable) Hardy space $\mathrm{H}^1(\reals^n).$ Second, the norm $\Vert b\Vert_{\mathrm{BMO}(\reals^n)}$ is the ``correct'' quantity controlling the boundedness of commutators $[T,b]=[T,M_{b}],$ where $M_{b}$ denotes (pointwise) multiplication by $b$ (called \emph{symbol} of the commutator) and $T$ is a Calder\'{o}n--Zygmound operator. This was proved for the Hilbert transform by Z.~Nehari \cite{Nehari_1957} and in full generality by R.~R.~Coifmann, Rochberg and G.~Weiss \cite{Coifman_Rochberg_Weiss_1976}. Moreover, the John--Nirenberg inequalities are an important intrinsic property of the space $\mathrm{BMO}(\reals^n).$
    
    B.~Muckenhoupt and R.~L.~Wheeden \cite{Muckenhoupt_Wheeden_BMO_1976} considered and studied the weighted BMO norm
    \begin{equation}
        \label{scalar weighted BMO}
        \Vert b\Vert_{\mathrm{BMO}(\nu)}\coloneq \sup_{Q}\frac{1}{\nu(Q)}\int_{Q}|b(x)-\langle b\rangle_{Q}|\mathdm(x),
    \end{equation}
    where the supremum ranges over all cubes $Q\subseteq\reals^n$ and $\nu$ is a $A_2$ weight on $\reals^n.$ A characterization of the two weighted boundedness of commutators in terms of a weighted BMO norm of the symbol was established by S.~Bloom \cite{Bloom_1985} for the Hilbert transform and later for arbitrary Calder\'{o}n--Zygmund operators by I.~Holmes, M.~Lacey and B.~Wick \cite{Holmes_Lacey_Wick_2016}. In the latter work two weighted versions of \eqref{scalar weighted BMO} were introduced and associated John--Nirenberg inequalities were established. These played an important role in the commutator estimates in \cite{Holmes_Lacey_Wick_2016}.

    The study of \emph{biparameter} BMO spaces on product spaces $\reals^n\times\reals^m$ was initiated by S.~Y.~A.~Chang \cite{Chang_1979} and R.~Fefferman \cite{RFefferman_1979}. Here, ``biparameter'' refers to invariance of the considered function spaces under rescaling each coordinate variable of the domain of definition separately. The papers \cite{Chang_1979} and \cite{RFefferman_1979} introduced and investigated the biparameter product BMO space $\text{BMO}(\reals\times\reals)$ consisting of all locally integrable functions $b$ on $\reals^2$ (considered as the product space $\reals\times\reals$) such that
    \begin{equation*}
        \Vert b\Vert_{\text{BMO}(\reals\times\reals)}\coloneq \sup_{\Omega}\bigg(\frac{1}{|\Omega|}\sum_{\substack{R\in\bfD\\ R\subseteq\Omega}}\left|(b,w_{R})\right|^2\bigg)^{1/2}<\infty,
    \end{equation*}
    where the supremum reanges over all (say) bounded Borel subsets $\Omega$ of $\reals^2$ with nonzero measure, $\bfD$ stands for the family of all dyadic rectangles of $\reals^2,$
    and $(w_{R})_{R\in\bfD}$ is some (mildly regular) wavelet system adapted to $\bfD.$ Here and below we denote
    \begin{equation*}
        (b,w_{R})\coloneq \int_{\reals^2}b(x)w_{R}(x)\mathdm(x).
    \end{equation*}
    The aforementioned works \cite{Chang_1979} and \cite{RFefferman_1979} established in particular that $\text{BMO}(\reals\times\reals)$ is the dual to the biparameter Hardy space $\mathrm{H}^1(\reals\times\reals).$ Moreover, a dyadic version of this product BMO space is the correct space for characterizing the boundedness of \emph{bicommutators} $[T_1,[T_2,b]],$ where $T_1,T_2$ are Haar multipliers, as proved by \cite{Blasco_Pott_2005}.

    A weighted version of the Chang--Fefferman product BMO space was introduced and studied by Holmes--Petermichl--Wick \cite{Holmes_Petermichl_Wick_2018} in the context of proving two weight upper bounds for biparameter praproducts. Namely, given a biparameter dyadic grid $\bfD$ in the product space $\reals^n\times\reals^m$ and a biparameter $\bfD$-dyadic $A_2$ weight $w$ on $\reals^n\times\reals^m,$ \cite{Holmes_Petermichl_Wick_2018} considers the \emph{dyadic} Bloom type product space $\BMOprodD(\nu)$ consisting of all locally integrable functions $b$ on $\reals^n\times\reals^m$ with the property
    \begin{equation}
        \label{scalar product BMO}
        \Vert b \Vert_{\BMOprodD(\nu)} \coloneq  \sup_{\Omega} \bigg( \frac{1}{\nu(\Omega)} \sum_{\substack{R\in\bfD(\Omega)\\\e\in\cE}} |b_R^{\e}|^2 \langle \nu^{-1} \rangle_{R} \bigg)^{1/2}<\infty,
    \end{equation}
    where the supremum ranges over all Lebesgue-measurable subsets $\Omega$ of $\reals^n\times\reals^m$ of nonzero finite measure.
    Section \ref{s:background} contains an explanation of the notation used here,
    including the definition of the Haar coefficients $b_R^{\e},$
    which can be found in expression \eqref{eq:HaarCoefDef}.
    In \cite{Holmes_Petermichl_Wick_2018} an $\mathrm{H}^1$-BMO duality type result was established in this setting, which played a crucial role in the proofs of the upper bounds there. More recently, a two-weight version of \eqref{scalar product BMO} was defined in \cite{Kakaroumpas_Soler_2022} and associated John--Nirenberg inequalities were established. These played an important role in \cite{Kakaroumpas_Soler_2022} for characterizing the two weight boundedness of bicommutators with Haar multipliers, extending the aforementioned result of \cite{Blasco_Pott_2005} to the two weight setting.

    \subsection{Main results}

    One of the main goals of this paper is to prove matrix weighted bounds for vector valued extensions of the Christ--Goldberg maximal function, which can be understood as a matrix weighted extension of the classical Fefferman--Stein vector valued inequalities for the Hardy--Littlewood maximal function.

    \begin{theorem}
        \label{thm:FeffermanSteinPointwiseWeightedMaximal}
        Let $1<p<\infty.$ Consider a $(d \times d)$ matrix $A_p$ weight $W$ and a sequence of vector valued functions $\{f_n\}^{\infty}_{n=1}.$
        Then, for each $1 < q < \infty$ it holds that
        \begin{equation*}
            \norm{\left(\sum_{n=1}^\infty |M_W f_n|^q\right)^{1/q}}{\Lp{p}}
            \leq C(n,d,p,q,[W]_{A_p}) \norm{\left(\sum_{n=1}^\infty |W(x)^{1/p} f_n|^q\right)^{1/q}}{\Lp{p}},
        \end{equation*}
        where
        \begin{equation*}
            C(n,d,p,q,[W]_{A_p}) = C(n,d,p,q) [W]_{A_p}^{\max\left\{\frac{1}{q-1},\frac{1}{p-1}\right\}},
        \end{equation*}
        and $M_{W}$ denotes the Christ--Goldberg maximal function corresponding to the weight $W$ and the exponent $p.$
    \end{theorem}

    We refer to Subsection~\ref{subsec:MatrixWeightedFeffermanStein} for a detailed explanation of the notation in \Cref{thm:FeffermanSteinPointwiseWeightedMaximal}. \Cref{thm:FeffermanSteinPointwiseWeightedMaximal} yields readily a similar estimate for the so called \emph{Christ--Goldberg auxiliary maximal operator}, as explained in \Cref{thm:FeffermanSteinReducingWeightedMaximal} below.
    
    We deduce \Cref{thm:FeffermanSteinPointwiseWeightedMaximal} from a general principle for establishing matrix weigh-ted bounds for vector valued extensions of operators acting on convex body valued functions, see \Cref{thm:VectorValuedEstimates} below. Our method for deducing such bounds is inspired from \cite{Book_Extrapolation}: we use an analog of the Rubio de Francia extrapolation theorem for matrix weights proved in \cite{Cruz_Uribe_Bownik_Extrapolation}, coupled with a trick of interpreting vector valued extensions of operators as operators whose values are convex body valued functions. It is worth noting that using the exact same method, the recent limited range extrapolation theorem for matrix weights proved in \cite{Kakaroumpas_Nguyen_Vardakis_2024} yields similar bounds as in \Cref{thm:VectorValuedEstimates} that are valid only for a limited range of exponents.
    
    Here it is important to remark one essential difference between the extrapolation theorem proved in~\cite{Cruz_Uribe_Bownik_Extrapolation} and what we actually use.
    The extrapolation for matrix weights and the other methods and techniques from \cite{Cruz_Uribe_Bownik_Extrapolation} (which also belong to the foundations of the work in \cite{Kakaroumpas_Nguyen_Vardakis_2024}) concern the setting of \emph{real} convex body valued functions. That is, functions taking values on the set of convex bodies in $\reals^d.$
    Our results and the techniques used to achieve them refer to \emph{complex} valued objects, namely $\complex^d$ vector valued functions and \emph{complex} convex body valued functions (taking values on the set of convex bodies in $\complex^d$).
    Most of the steps in translating the results from the real setting to the complex setting are immediate, thanks to the fact that $\complex^d$ and $\reals^{2d}$ share the same structure as topological spaces, metric spaces, measure spaces and real vector spaces.
    However, not all steps are covered by these considerations.
    In particular, the main difficulty is that it is not clear that,
    given a measurably parametrized family of norms over $\complex^d$
    (such as the ones defined by a matrix weight), there is a measurable choice of a reducing operator for such a family.
    This was already proved separately in~\cite{DKPS2024}.
    Nonetheless, in Appendix~\ref{s:appendix} we cover the proof of Theorem~\ref{thm:Extrapolation}.
    The more obvious steps are only mentioned together with the reason why they hold
    exactly as in the real variable case,
    while we devote more detail to the steps that are not immediate.

    An immediate application of \Cref{thm:FeffermanSteinPointwiseWeightedMaximal}
    and \Cref{thm:FeffermanSteinReducingWeightedMaximal}, which is the analogue for the Christ--Goldberg auxiliary maximal operator, is the following.
    In~\cite{DKPS2024}, the authors show matrix weighted $\Lp{p}$ bounds
    for Journé operators in the biparameter setting.
    However, the bounds in that article are only complete in the case $p = 2.$
    Even though that article also contains bounds for all $1 < p < \infty,$
    for $p \neq 2$ these only hold for the particular class of paraproduct free
    Journé operators.
    The general bounds can be obtained using extrapolation for biparameter matrix weights
    (see~\cite{Vuorinen2024}).
    Nonetheless, one can also apply \Cref{thm:FeffermanSteinPointwiseWeightedMaximal}
    and \Cref{thm:FeffermanSteinReducingWeightedMaximal}
    as it was outlined already in~\cite[Section~8]{DKPS2024} to get complete bounds
    for all biparameter Journé operators,
    avoiding in this way the use of extrapolation for biparameter matrix weights.
    
    In this paper we focus on the application of \Cref{thm:FeffermanSteinPointwiseWeightedMaximal} for setting up the foundations of a theory of two matrix weighted product BMO. Namely, let $\bfD=\cD^1\times\cD^2$ be any product dyadic grid in the product space $\reals^n\times\reals^m.$ Let $1<p<\infty,$ and let $U,V$ be biparameter $(d\times d)$ matrix $\bfD$-dyadic $A_p$ weights on $\reals^n\times\reals^m.$ Let $B=\{B_{R}^{\e}\}_{\substack{R\in\bfD\\\e\in\cE}}$ be any sequence in $\mathrm{M}_{d}(\complex).$ We define
    \begin{equation*}
        \Vert B\Vert_{\BMOprodD(U,V,p)}\coloneq \sup_{\Omega}\frac{1}{|\Omega|^{1/2}}\bigg(\sum_{\substack{R\in\bfD(\Omega)\\\e\in\cE}}|\cV_{R}B_{R}^{\e}\cU_{R}^{-1}|^2\bigg)^{1/2},
    \end{equation*}
    where the supremum ranges over all Lebesgue-measurable subsets $\Omega$ of $\reals^{n+m}$ of nonzero finite measure. This definition is an extension of one of the equivalent definitions for the space of two matrix weighted one-parameter BMO, whose study was initiated in \cite{Isralowitz_Kwon_Pott2017}, \cite{Isralowitz_2017} and culminated in \cite{Isralowitz_Pott_Treil_2022}. Moreover, for every sequence $\Phi=\{\Phi_{R}^{\e}\}_{\substack{R\in\bfD\\\e\in\cE}}$ in $\mathrm{M}_{d}(\complex),$ we define
    \begin{equation*}
        \Vert\Phi\Vert_{\mathrm{H}^1_{\bfD}(U,V,p)}\coloneq \bigg\Vert\bigg(\sum_{\substack{R\in\bfD\\\e\in\cE}}|V^{-1/p}\Phi_{R}^{\e}\cU_{R}|^2\frac{\1_{R}}{|R|}\bigg)^{1/2}\bigg\Vert_{\Lp{1}(\reals^{n+m})},
    \end{equation*}
    which is the direct biparameter analog of the one parameter two matrix weighted $\mathrm{H}^1$ norm from \cite{Isralowitz_2017}.
    In this context, we prove the following result.

    \begin{theorem}
        \label{thm:H1BMOduality}
        Let $U,V,p,\bfD$ be as above. For any $B\in\emph{BMO}_{\emph{prod},\bfD}(U,V,p),$ the linear functional $\ell_{B}:\mathrm{H}^1_{\bfD}(U,V,p)\rightarrow\complex$ given by
        \begin{equation*}
            \ell_{B}(\Phi)\coloneq \sum_{\substack{R\in\bfD\\\e\in\cE}}\mathrm{tr}((B_{R}^{\e})^{\ast}\Phi_{R}^{\e}),\quad \Phi\in \mathrm{H}^1_{\bfD}(U,V,p)
        \end{equation*}
        is well-defined and bounded with $\Vert\ell_{B}\Vert\sim\Vert B\Vert_{\emph{BMO}_{\emph{prod},\bfD}(U,V,p)},$ where the implicit constants depend only on $n,m,p,d$ and $[V]_{A_p,\bfD}.$ Conversely, for every bounded linear functional $\ell$ on $\mathrm{H}^1_{\bfD}(U,V,p)$ there is $B\in\emph{BMO}_{\emph{prod},\bfD}(U,V,p)$ with $\ell=\ell_{B}.$
    \end{theorem}

    We note that our proof of Theorem \ref{thm:H1BMOduality} trivially works also in the one-parameter setting, thus answering to the positive the question posed in \cite{Isralowitz_2017} about the extension of the one parameter two matrix weighted $\mathrm{H}^1$-BMO duality proved there from $p=2$ to arbitrary exponents $1<p<\infty.$

    The duality result of Theorem \ref{thm:H1BMOduality} coupled with Theorem \ref{thm:FeffermanSteinPointwiseWeightedMaximal} yields readily two-matrix weighted bounds for biparameter paraproducts, see Proposition \ref{prop:pureparaproducts} and Proposition \ref{prop:mixedparaproducts} below. These in turn yield some two matrix weighted upper bounds for bicommutators, see Subsection~\ref{subsec:bicommutators} below.

    The remainder of this article is structured as follows.
    In Section~\ref{s:background}, we review the objects that we work with
    and their relevant properties for our results.
    Section~\ref{sec:extrapolation_vector_valued} is focused on vector valued estimates.
    There we state the relevant result of~\cite{Cruz_Uribe_Bownik_Extrapolation} in the complex setting,
    we outline very briefly its proof (deferring the more technical details to the Appendix),
    and then we use it to show vector valued estimates for general families of extrapolation pairs.
    This general result is followed by vector valued estimates for matrix weighted maximal operators.
    Namely, we show such estimates both for the Christ--Goldberg maximal operator
    and for the Christ--Goldberg auxiliary maximal operator.
    In Section~\ref{s:MatrixWeighedProductBMO} we develop a two-matrix weighted $\mathrm{BMO}$ theory.
    First, we define a two-matrix weighted matrix $\mathrm{BMO}$ space,
    and then we show that it is dual to a two-matrix weighted matrix $\mathrm{H}^1$
    space already considered in the literature.
    Next, we use the previously developed vector valued estimates to prove
    norm bounds for biparameter matrix paraproducts,
    relating them to the $\mathrm{BMO}$ norm of their symbol.
    We conclude that section with a sketch on how our results can be also
    applied to obtaining bounds for bicommutators with Haar multipliers and
    multiplication with a matrix symbol.
    Lastly, in Appendix~\ref{s:appendix} we include the more technical details
    on how to obtain the complex version of Bownik and Cruz-Uribe extrapolation
    for matrix weights found in~\cite{Cruz_Uribe_Bownik_Extrapolation}.

    \subsection*{Acknowledgements}
    The authors are grateful to S.~Treil for suggesting the interpretation of the $q$-function as an operator whose values are convex body valued functions for any $1<q<\infty,$ inspiring the interpretation of vector valued extensions of operators as operators whose values are convex body valued functions used in this paper.
    In addition, the authors are indebted to both Marcin Bownik and David Cruz-Uribe
    for carefully reading this article and for providing valuable feedback
    that improved the clarity and readability of this work.
    Finally, the authors are also thankful to the anonymous referee,
    who made very useful suggestions and comments to improve the clarity of this article.

    \section{Background}
    \label{s:background}

  
    \subsection{Notation for dyadic grids and Haar systems}
    We briefly review the notation that we use for product dyadic grids.
    Consider two different dyadic grids $\cD^1$ and $\cD^2$ on $\reals^n$ and $\reals^m$ respectively;
    that is, both $\cD^1$ and $\cD^2$ are collections of dyadic \emph{cubes}.
    Then we say that $\bfD$ is the \emph{product dyadic grid} given by $\cD^1$ and $\cD^2$ on $\reals^{n} \times \reals^m,$
    denoted by $\bfD = \cD^1 \times \cD^2,$
    if $\bfD$ is the collection of \emph{dyadic rectangles} of the form $R_1 \times R_2$ with $R_j \in \cD^j,$ $j = 1,2.$

    Given a dyadic interval $I$ on $\reals,$ we denote the cancellative and noncancelative Haar functions on $I$ by
    \begin{equation*}
        h^{0}_{I}\coloneq \frac{\1_{I_{+}}-\1_{I_{-}}}{\sqrt{|I|}},\qquad h^{1}_{I}\coloneq \frac{\1_{I}}{\sqrt{|I|}}.
    \end{equation*}
    Consider a dyadic grid $\cD$ on $\reals^n$ and let $\cE = \{0,1\}^n \setminus \{(1,\ldots,1)\}.$
    The Haar basis on $\reals^n$ adapted to $\cD$ is the set of functions indexed by $Q \in \cD$ and $\e \in \cE$ given by
    \begin{equation*}
        h^{\e}_{Q}(x) \coloneq  h^{\e_1}_{I_1}(x_1) \dots h^{\e_n}_{I_n}(x_n),
    \end{equation*}
    where $Q = I_1 \times \dots \times I_n$ ($\cE$ is called the signature set of signatures on $\reals^n$).
    Similarly, in the biparameter context, consider a product dyadic grid $\bfD = \cD^1 \times \cD^2$ and
    the set of biparameter signatures $\cE = \cE^1 \times \cE^2,$ where $\cE^1$ and $\cE^2$ are respectively the signature sets on $\reals^n$ and $\reals^m.$
    The biparameter Haar basis on $\reals^{n} \times \reals^m$ adapted to $\bfD$ is the set of functions
    \begin{equation*}
        h^{\e}_{R}(x) \coloneq  h^{\e_1}_{R_1}(x_1) \cdot h^{\e_2}_{R_2}(x_2),
    \end{equation*}
    with $R = R_1 \times R_2 \in \bfD$ and $\e = (\e_1,\e_2) \in \cE.$

    For a given locally integrable function $f,$ we will denote the Haar coefficient of $f$ corresponding to $h^{\e}_{R}$ by $f^{\e}_R$
    (with analogous notation for the one parameter setting).
    In other words, it is the coefficient defined by
    \begin{equation}
        \label{eq:HaarCoefDef}
        f^{\e}_R =  (f, h^{\e}_{R}).
    \end{equation}
    We will denote the average of $f$ on a rectangle $R$ by $\langle f \rangle_R.$

    \subsection{Reducing operators}
    \label{subsec:ReducingOperators}
    Let $1<p<\infty$. Let $E$ be a bounded measurable subset of $\reals^n$ of nonzero measure. Let $W$ be an integrable $\mathrm{M}_{d}(\complex)$-valued function on $E$ taking a.e.~positive-definite values.
    Here, by a positive-definite matrix $M \in \mathrm{M}_{d}(\complex)$
    we mean a matrix such that $\langle M e, e \rangle \geq 0$ for every $e \in \complex^d.$
    In particular, since we are focusing on the space $\complex^d$, note that this implies that a positive-definite matrix $M$ is also Hermitian.
    It is proved in \cite[Proposition 1.2]{Goldberg_2003} that there exists a (not necessarily unique) positive-definite matrix $\cW_{E}\in \mathrm{M}_{d}(\complex)$, called \emph{reducing operator} of $W$ over $E$ with respect to the exponent $p$, such that
    \begin{equation*}
        \langle|W^{1/p}e|\rangle_{E,p}\leq
        |\cW_{E}\, e| \leq
        \sqrt{d}\langle|W^{1/p}e|\rangle_{E,p}
    \end{equation*}
    for all $e \in \complex^d,$ where $\langle f\rangle_{E,p}$ denotes the $\mathrm{L}^{p}$ average of a scalar valued function $f$ over $E$. If the function $W' \coloneq W^{-1/(p-1)}$ is also integrable over $E,$ then we let $\cW'_{E}$ be the reducing matrix of $W'$ over $E$ corresponding to the exponent $p' \coloneq p/(p-1)$. For a detailed exposition of reducing operators we refer for example to \cite{DKPS2024}.
        

    
    Let $E,F$ be bounded measurable subsets of $\reals^n,\reals^m$ respectively of nonzero measure. Let $1<p<\infty.$ Let $W$ be an integrable $\mathrm{M}_{d}(\complex)$-valued function on $E\times F$ taking a.e.~positive definite values. For a.e.~$x_1\in E,$ denote by $\cW_{x_1,F}$ the reducing operator of $W_{x_1}(x_2)\coloneq W(x_1,x_2)$ over $F$ with respect to the exponent $p.$
    It is proved in \cite{DKPS2024} (see also \cite{Cruz_Uribe_Bownik_Extrapolation}) that one can choose the reducing operator $\cW_{x_1,F}$ in a way that is measurable in $x_1.$
    Moreover, it is shown in \cite{DKPS2024} that
    \begin{equation}
        \label{iterate reducing operators}
        |\cW_{F,E}e|\sim_{p,d}|\cW_{E\times F}e|,\qquad\forall e\in\complex^d,
    \end{equation}
    where $\cW_{F,E}$ is the reducing operator of $W_{F}(x_1)\coloneq \cW_{x_1,F}^{p}$ over $E$ with respect to the exponent $p,$ and $\cW_{E\times F}$ is the reducing operator of $W$ over $E\times F$ with respect to the exponent $p.$

    \subsection{Matrix \texorpdfstring{$A_p$}{Ap} weights}
    \label{sec:MatrixAp}

    A function $W$ on $\reals^n$ is said to be a \emph{$d\times d$-matrix valued weight,} or just a matrix weight, if
    it is a locally integrable $\mathrm{M}_{d}(\complex)$-valued function such that $W(x)$ is a positive-definite matrix for a.e.~$x\in\reals^n.$ Given $1 < p < \infty$ we define the norm
    \begin{equation*}
        \Vert F\Vert_{\Lp{p}(W)} \coloneq \Vert |W^{1/p}F|\Vert_{\mathrm{L}^{p}}
    \end{equation*}
    for all $\mathrm{M}_{d}(\complex)$-valued measurable functions $F$ on $\reals^n.$
    For $p = \infty,$ this definition needs a slight modification.
    Namely, we define the norm
    \begin{equation*}
        \Vert F\Vert_{\Lp{\infty}(W)} \coloneq \Vert |WF|\Vert_{\mathrm{L}^{\infty}}
    \end{equation*}
    for the same class of matrix valued functions as before.

    In this work we will be considering in particular \emph{one-parameter}, respectively \emph{biparameter ($d\times d$) matrix valued $A_p$ weights} on $\reals^n,$ respectively on $\reals^n\times\reals^m,$ for $1<p<\infty.$
    Following~\cite{Roudenko2002}, we say that a matrix weight on $\reals^n$ is a one-parameter $A_p$ weight if
    \begin{equation}
        \label{eq:RoudenkoCondition}
        \sup_{R}\frac{1}{|R|}\int_{R}\left(\frac{1}{|R|}\int_{R}|W(x)^{1/p}W(y)^{-1/p}|^{p'}\,\mathrm{d}y\right)^{p/p'}\,\mathrm{d}x
        < \infty,
    \end{equation}
    where the supremum ranges over all bounded cubes $R$ with sides parallel to the axes.
    We say that a matrix weight on $\reals^n \times \reals^m$ is a biparameter
    $A_p$ weight if the same condition holds
    when the supremum ranges over rectangles $R$  which are cartesian products
    of cubes on $\reals^n$ and on $\reals^m$ with sides parallel to the axes.
    For $p = 1,$ note that condition~\eqref{eq:RoudenkoCondition} becomes
    \begin{equation*}
        \sup_{R} \esssup_{x \in R} \frac{1}{|R|}\int_{R} |W(x)^{-1}W(y)|\,\mathrm{d}y
        < \infty,
    \end{equation*}
    with the same considerations as before distinguishing the one-parameter and biparameter cases.
    For $p = \infty,$ following~\cite{Cruz_Uribe_Bownik_Extrapolation},
    we say that a matrix weight is an $A_\infty$ weight if
    \begin{equation*}
        \sup_{R} \esssup_{x \in R} \frac{1}{|R|}\int_{R} |W(x)W(y)^{-1}|\,\mathrm{d}y
        < \infty,
    \end{equation*}
    with identical considerations distinguishing the one-parameter and biparameter cases.
    For a more detailed explanation for these definitions
    and their relation to $A_p$ norm functions, see~\cite[Section~6]{Cruz_Uribe_Bownik_Extrapolation}.
    Also, for a general overview on the topics of matrix weights,
    convex body valued functions, matrix extrapolation
    and their relation to the scalar theory, we refer the reader to
    \cite{CruzUribe2025}.
    
    In the second part of this work, dyadic versions of these weights adapted to dyadic grids in $\reals^n,$ respectively product dyadic grids in $\reals^n\times\reals^m,$ will also be of major importance. Since the relevant definitions are by now standard, we omit repeating them and instead refer the reader to the detailed exposition in \cite[Subsection 3.3]{DKPS2024}. Let us only stress for later convenience the following estimates from \cite{DKPS2024}. Let $\bfD=\cD^1\times\cD^2$ be a product dyadic grid in $\reals^n\times\reals^m$ and let $W$ be a $(d\times d)$ matrix $\bfD$-dyadic biparameter $A_p$ weight on $\reals^n\times\reals^m.$
    Then
    \begin{equation}
        \label{uniform slicing variable}
        [W_{x_1}]_{A_p,\cD^2}\lesssim_{p,d}[W]_{A_p,\bfD},\quad\text{for a.e. }x_1\in\reals^n.
    \end{equation}
    Moreover, if $Q$ is any cube in $\cD^2,$ then
    \begin{equation}
        \label{uniform domination of characteristics of averages}
        [W_{Q}]_{A_p,\cD^1}\lesssim_{p,d}[W]_{A_p,\bfD}.
    \end{equation}

    \subsection{Convex body valued functions}
    \label{sec:ConvexBodyValuedFunctions}
    Here we consider functions taking values in the collection of closed bounded symmetric convex sets of $\complex^d.$
    Take into account that by symmetric here we mean \emph{complex symmetric}.
    That is, a set $A \subseteq \complex^d$ is complex symmetric (or just symmetric in this article) if
    for every $u \in A$ and every $\lambda \in \complex$ with $|\lambda| = 1$ it is also the case that $\lambda u \in A.$
    Of course, if the set $A$ is convex in addition to symmetric, it will also be the case that,
    for every $u \in A$ and every $\lambda \in \complex$ with $|\lambda| \leq 1,$ also $\lambda u \in A.$ In other words, the ``symmetric convex sets'' as we have defined them are precisely the balanced convex sets.
    We will denote the set of closed subsets of $\complex^d$ by $\cK(\complex^d),$
    or just $\cK$ when the dimension of the ambient space is clear by the context.
    In addition, we define a \emph{convex body} to be a closed bounded convex and symmetric subset of $\complex^d.$
    We will use the symbol $\convex(\complex^d)$ to denote the set of convex bodies on $\complex^d$
    and, whenever the dimension of the ambient space is unambiguous, simply by $\convex.$
    Recall that if $K \in \convex(\complex^d)$ has nonempty interior,
    then there exists a unique ellipsoid $E \subseteq K$ with maximal volume.
    This ellipsoid is called the \emph{John ellipsoid} and it satisfies
    \begin{equation*}
        E \subseteq K \subseteq \sqrt{d} E.
    \end{equation*}
    
    We focus now on functions $F\colon \reals^n \rightarrow \convex(\complex^d)$
    and we gather the definitions and basic properties that we will need for our results.
    Some of these definitions can be found in more general forms in the texts that we cite;
    we will restrict though to the cases that are of interest to us to avoid an excess of concepts.
    For an introduction to such functions, their properties and how to define their integrals
    see~\cite{Cruz_Uribe_Bownik_Extrapolation, Cruz_Uribe_Extrapolation} and
    for a detailed exposition see~\cite{Aubin_2009}.

    Since throughout the article we will consider both functions taking values in $\complex^d$
    and functions taking values in $\convex,$
    we will use a typographic convention to avoid confusion.
    We will denote functions taking values in $\complex^d$ with lowercase letters $f,g,h,\ldots$
    On the other hand, we will use uppercase letters $F,G,H,\ldots$ to denote functions taking values in $\convex(\complex^d).$
    In any case, we will also explicitly state the target space of the functions we use.

    Given a set $K \subseteq \complex^d,$ let us define its norm by
    \begin{equation*}
        |K| \coloneq \sup \{|v|\colon v \in K\}.
    \end{equation*}
    For a matrix $A$ we will denote its usual norm (given by its largest singular value) by $|A|.$
    The action of matrix weights on convex body valued functions is given by the next definition.
    Given a convex body $K$ and a positive definite matrix $A,$ define the product
    \begin{equation*}
        A K \coloneq \{A u\colon u \in K\}
    \end{equation*}
    and observe that $AK$ will also be a convex body.

    We will say that a function $F\colon \reals^n \rightarrow \convex(\complex^d)$ is \emph{measurable}
    if for every open set $U \subseteq \complex^d$ it holds that the set
    \begin{equation*}
        F^{-1}(U) \coloneq \{x\in\reals^n\colon F(x) \cap U \neq \emptyset\}
    \end{equation*}
    is measurable (in the sense of Lebesgue).
    A convex body valued function $F$ is measurable if and only if there exists a sequence $\{f_k\}_{k \geq 1}$
    of measurable functions $f_k\colon \reals^n \rightarrow \complex^d$ such that
    \begin{equation}
        \label{eq:SelectionFunctionMeasurability}
        F(x) = \overline{\{f_k(x)\colon k \geq 1\}}
    \end{equation}
    for every $x \in \reals^n$ (see~\cite[Theorem~8.1.4]{Aubin_2009}).
    The functions that form such sequences will be called \emph{selection functions} for $F.$
    In addition, the set of all selection functions for $F$ is denoted by $S^0(F).$
    Observe that this is the same as the set of all measurable functions $f$ such that $f(x) \in F(x)$
    for every $x \in \reals^n.$
    Here we restrict to functions taking values on $\convex(\complex^d)$ due to the applications that we consider later.
    Nonetheless, the previous definitions and concepts apply verbatim to functions taking values on $\cK(\complex^d).$
    When one is interested in the norm of $F,$ it is possible to restrict to selection functions, as the following lemma shows
    (see~\cite[Lemma~3.9]{Cruz_Uribe_Bownik_Extrapolation}).
    We include its proof for completeness, although the arguments of Bownik and Cruz-Uribe
    are equally valid for complex convex bodies in this case.
    \begin{lemma}
        \label{lemma:ModulusSelectionFunction}
        Consider a measurable function $F\colon \reals^n \rightarrow \convex(\complex^d).$
        Then there exists $f \in S^0(F)$ such that
        \begin{equation*}
            |f(x)| = |F(x)|
        \end{equation*}
        for all $x \in \reals^n.$
    \end{lemma}
    \begin{proof}
        Through this proof, we identify $\complex^d$ with $\reals^{2d}$
        in the usual way.
        Take the sequence $\{f_k\}$ of selection functions satisfying~\eqref{eq:SelectionFunctionMeasurability}.
        Then, the function given by
        \begin{equation*}
            g_0(x) = \sup\{|v|\colon v \in F(x)\} = \sup\{|f_k(x)|\colon k \geq 1\}
        \end{equation*}
        is also measurable.
        This allows us to define the function
        \begin{equation*}
            F_0(x) = \{v \in F(x)\colon |v| = g_0(x)\} = F(x) \cap (g_0(x)\complexsphere{d})
        \end{equation*}
        taking values on $\cK(\complex^d),$
        where $\complexsphere{d} = \{v \in \complex^d\colon |v| = 1\}$ denotes the complex
        $(d-1)$-dimensional sphere.
        Since both $F$ and $g_0 \complexsphere{d}$ are measurable, so is $F_0$
        (see~\cite[Theorems~8.2.2 and~8.2.4]{Aubin_2009}).

        It is only left to choose $v(x) \in F_0(x)$ in a measurable way.
        This is done by choosing maximal vectors $v$ in every other real coordinate iteratively (in the $\reals^{2d}$ sense).
        Let $P_1$ be the (continuous) projection onto the first complex coordinate and define
        \begin{equation*}
            g_1(x) = \sup_k |P_1(f_k(x))|.
        \end{equation*}
        The function $g_1\colon \Omega \rightarrow [0,\infty)$ is then measurable.
        Next, use this function and the symmetry of $F(x)$ to define $F_1\colon \Omega \rightarrow \cK(\complex^d)$ as
        \begin{equation*}
            F_1(x) = \{v \in F_0(x)\colon P_1(v) = g_1(x)\}
                   = F_0(x) \cap (\{g_1(x)\} \times \complex^{d-1}).
        \end{equation*}
        As before, $F_1$ is also a measurable function.
        Assume that we have defined $F_j$ measurable and taking values on $\cK(\complex^d).$
        Define $F_{j+1}\colon \Omega \rightarrow \cK(\complex^d)$ by taking the set of points
        of $F_j(x)$ with maximal modulus of the $j+1$ complex coordinate, positive real part and zero imaginary part for the same coordinate.
        Also $F_{j+1}$ will be measurable because of the same reasons as before.
        Eventually, we get the measurable function $F_d\colon \Omega \rightarrow \cK(\complex^d)$ and $F_d(x)$ must be a singleton for every $x \in \Omega,$
        this is $F_d(x) =\{v(x)\}$ for some $v(x) \in \complex^d,$
        because of the maximality of the modulus of each complex coordinate
        and the restriction of having each complex coordinate lying on the positive real axis.
    \end{proof}

    In order to define integrals of convex body valued functions, consider first the set
    \begin{equation*}
        S^1(F) \coloneq \{f \in S^0(F)\colon f \in \Lp{1}(\reals^n)\}
    \end{equation*}
    of integrable selection functions for $F.$
    The \emph{Aumann integral} of $F$ is then defined as
    \begin{equation*}
        \int_{\reals^n} F(x)\, \mathdm(x) = \left\{\int_{\reals^n} f(x)\, \mathdm(x)\colon f \in S^1(F)\right\}.
    \end{equation*}
    In this work, we will restrict to \emph{integrably bounded} convex body valued functions,
    that is to functions $F$ such that $|F(x)|$ is integrable.
    In particular, in this case $S^0(F) = S^1(F).$
    If $|F(x)|$ is only locally integrable, we will say that $F$ is \emph{locally integrably bounded}.
    For further convenience, given a cube $Q \subseteq \reals^n,$ we define the averaging operator $A_Q$ by
    \begin{equation*}
        A_QF(x) \coloneq \frac{\1_Q(x)}{|Q|} \int_{Q} F(y)\, \mathdm(y),
    \end{equation*}
    where $F$ is a locally integrably bounded convex body valued function.

    Given $1 \leq p < \infty,$ we define the Lebesgue space of convex body valued functions $\Lp{p}(\reals^n,\convex(\complex^d)),$
    or just $\Lp{p}$ when there is no ambiguity, as the set of functions $F\colon \reals^n \rightarrow \convex(\complex^d)$ such that
    \begin{equation*}
        \norm{F}{\Lp{p}} \coloneq \left(\int_{\reals^n}|F(x)|^p\, \mathdm(x)\right)^{1/p} < \infty.
    \end{equation*}
    The space of convex body valued functions $\Lp{\infty}(\reals^n,\convex(\complex^d)),$ or just $\Lp{\infty},$ is defined as
    the set of convex body valued functions $F$ for which
    \begin{equation*}
        \norm{F}{\Lp{\infty}} \coloneq \esssup \{|F(x)|\colon x \in \reals^n\} < \infty.
    \end{equation*}
    Given a $(d \times d)$-matrix weight $W,$ we define the weighted Lebesgue space $\Lp{p}(W)$ of convex body valued functions
    as the set of functions $F$ for which $W(x)^{1/p} F(x) \in \Lp{p}$ if $1\leq p<\infty$ and $W(x)F(x)\in\Lp{\infty}$ if $p=\infty$. Finally, we denote $\Vert F\Vert_{\Lp{p}(W)} := \Vert W^{1/p} F\Vert_{\Lp{p}}$ if $1\leq p<\infty$ and $\Vert F\Vert_{\Lp{\infty}(W)} := \Vert W F\Vert_{\Lp{\infty}}$ if $1\leq p<\infty$.

    \section{Vector valued extensions of operators on convex body valued functions}
    \label{sec:extrapolation_vector_valued}
    The aim of this section is to prove vector valued estimates for convex body valued functions with matrix weights,
    as well as to apply it to vector valued estimates for the matrix weighted maximal operators for vector valued functions
    (see Subsection~\ref{subsec:MatrixWeightedFeffermanStein} for the definitions and details).
    This will be a consequence of an extrapolation result on spaces of functions with values on $\convex$
    and it is a modification of the analogous result for matrix weights and vector valued functions due to
    Bownik and Cruz-Uribe (see~\cite[Section~9]{Cruz_Uribe_Bownik_Extrapolation}).
    Before stating the extrapolation theorem, we introduce the concept of families of extrapolation pairs
    following the convention in~\cite{Book_Extrapolation}.
    A family of extrapolation pairs $\cF$ is a collection of pairs $(F,G)$ of measurable functions
    taking values in $\convex$ such that neither $F$ nor $G$ is equal to $\{0\}$ almost everywhere.
    In addition, given such a family $\cF,$ we call each element $(F,G) \in \cF$ an extrapolation pair.
    We are interested in inequalities of the form
    \begin{equation*}
        \norm{F}{\Lp{p}(W)} \leq C \norm{G}{\Lp{p}(W)},\qquad (F,G) \in \cF,
    \end{equation*}
    in the sense that this holds for all pairs $(F,G) \in \cF$ for which the left-hand side is finite
    and with the constant $C$ depending on the characteristic $[W]_{A_p}$ but not on the particular weight $W.$

    \begin{lettertheorem}[Bownik,~Cruz-Uribe~\cite{Cruz_Uribe_Bownik_Extrapolation}]
        \label{thm:Extrapolation}
        Consider a family of extrapolation pairs $\cF.$
        Suppose that for some $p_0,$ $1 \leq p_0 \leq \infty,$ there exists an increasing function $C_{p_0}$
        such that for every matrix weight $W \in A_{p_0}$ it holds that
        \begin{equation}
            \label{eq:BootstrapInequalities}
            \norm{F}{\Lp{p_0}(W)} \leq C_{p_0}([W]_{A_{p_0}}) \norm{G}{\Lp{p_0}(W)},\qquad (F,G) \in \cF.
        \end{equation}
        Then, for all $1 < p < \infty$ and for every matrix weight $W \in A_p$ it holds that
        \begin{equation*}
            \norm{F}{\Lp{p}(W)} \leq C_{p}(p_0,n,d,[W]_{A_p}) \norm{G}{\Lp{p}(W)},\qquad (F,G) \in \cF,
        \end{equation*}
        where
        \begin{equation*}
            C_p(p_0,n,d,[W]_{A_p}) = C(p,p_0) C_{p_0}\left(C(n,d,p,p_0) [W]_{A_p}^{\max\left\{1,\frac{p_0-1}{p-1}\right\}}\right)
        \end{equation*}
        if $p_0<\infty,$ and
        \begin{equation*}
            C_p(p_0,n,d,[W]_{A_p}) = C(p,p_0) C_{p_0}\left(C(n,d,p,p_0) [W]_{A_p}^{\frac{1}{p-1}}\right)
        \end{equation*}
        if $p_0=\infty.$
    \end{lettertheorem}

    This statement is the complex variable version of~\cite[Theorem~9.1]{Cruz_Uribe_Bownik_Extrapolation}.
    Also, observe that~\cite[Theorem~9.1]{Cruz_Uribe_Bownik_Extrapolation}
    refers to real \emph{vector valued} functions, but Theorem~\ref{thm:Extrapolation}
    is stated directly for complex \emph{convex body} valued functions.
    The reason is that, in~\cite{Cruz_Uribe_Bownik_Extrapolation},
    given real vector valued functions $f$ and $g$ Bownik and Cruz-Uribe
    define auxiliary real convex body valued functions $F$ and $G$ by taking
    the segments with endpoints $\pm f$ and $\pm g$ respectively,
    and the whole proof is done for these convex body valued functions.
    For the sake of generality, we opted to use convex body valued functions since the beginning.
    The proof of Theorem~\ref{thm:Extrapolation} follows that
    of~\cite[Theorem~9.1]{Cruz_Uribe_Bownik_Extrapolation} almost verbatim
    and we do not include it here.
    However, there are some minor differences and a couple of technical details
    that require some care.
    The main difficulty in translating the real variable proof to the complex variable one is the following.
    In~\cite{Cruz_Uribe_Bownik_Extrapolation} Bownik and Cruz-Uribe consider a family of norms
    on $\reals^d$ parametrized in a measurable way,
    such as the one given by a real (almost everywhere positive definite) matrix weight,
    and they include a proof of the fact that one can choose
    a family of reducing operators for such a family of norms.
    The proof of the analogous fact in the complex setting appears already
    in~\cite[Section~9]{DKPS2024}.
    We defer the discussion on Theorem~\ref{thm:Extrapolation} to Appendix~\ref{s:appendix}.
    There we review the proof of~\cite[Theorem~9.1]{Cruz_Uribe_Bownik_Extrapolation}.
    Most of the steps are just commented briefly together with an explanation of why
    they hold all the same in the complex setting.
    We devote more detail though to those steps that are not immediate,
    namely the measurability of reducing operators.
    We also refer the reader to~\cite[Chapter~I.3]{Book_Extrapolation} for a detailed exposition on the topic in the classical setting.

    The main result of this section follows from Theorem~\ref{thm:Extrapolation}.
    This is a vector valued estimate for sequences $\{(F_n,G_n)\} \subseteq \cF$
    of a given family of extrapolation pairs $\cF.$

    \begin{theorem}[Vector valued estimates for families of extrapolation pairs]
        \label{thm:VectorValuedEstimates}
        Consider a family of extrapolation pairs $\cF.$
        Suppose that for some $1 \leq p_0 \leq \infty$ there exists an increasing function $C_{p_0}$
        such that for every matrix weight $W \in A_{p_0}$ inequalities~\eqref{eq:BootstrapInequalities} hold.
        Then, for any $p$ and $q,$ $1 < p,q < \infty,$ and for every matrix weight $W \in A_p$ it holds that
        \begin{equation*}
            \norm{\left(\sum_{n=1}^\infty |W^{1/p}F_n|^q\right)^{1/q}}{\Lp{p}}
            \leq C(n,d,p,p_0,q,[W]_{A_p}) \norm{\left(\sum_{n=1}^\infty |W^{1/p}G_n|^q\right)^{1/q}}{\Lp{p}},
        \end{equation*}
        where $\{(F_n,G_n)\} \subseteq \cF$ and
        \begin{align*}
            C(n,d,p_0,p,q,[W]_{A_p}) = &C(d,p,q)\\
                &C_{p_0}\left(C(n,d,p_0,p,q) [W]_{A_p}^{\left(\max\left\{1,\frac{p_0-1}{q-1}\right\}\max\left\{1,\frac{q-1}{p-1}\right\}\right)}\right)
        \end{align*}
        if $p_0 < \infty,$ and
        \begin{equation*}
            C(n,d,p_0,p,q,[W]_{A_p}) = C(d,p,q) C_{p_0}\left(C(n,d,p_0,p,q) [W]_{A_p}^{\left(\frac{1}{q-1}\max\left\{1,\frac{q-1}{p-1}\right\}\right)}\right)
        \end{equation*}
        if $p_0 = \infty.$
    \end{theorem}

    Recall that one can define the Minkowski addition of $A,B \in \convex$ (and for subsets of $\complex^d$ in general)
    by
    \begin{equation*}
        A + B \coloneq \{a + b\colon a \in A, b \in B\}.
    \end{equation*}
    It is easy to check that if $A,B \in \convex,$ then also $A+B \in \convex.$
    One can extend the Minkowski addition of two convex body to that of $N$ convex bodies by induction, with $N \geq 2.$
    We can also define the scalar multiplication, given $A \in \convex$ and $\lambda \in \complex,$ by
    \begin{equation*}
        \lambda A \coloneq \{\lambda a\colon a \in A\},
    \end{equation*}
    and also $\lambda A \in \convex.$
    Keep in mind that these operations do not define a vector space structure, since the Minkowski addition has no inverse.
    
    Given $1 \leq q < \infty$ and a sequence $\{K_n\}_{n=1}^\infty \subseteq \convex$ such that
    \begin{equation}
        \label{eq:ConvexBodiesLittleLpSummability}
        \sum_{n=1}^\infty |K_n|^q < \infty,
    \end{equation}
    we define the infinite $\lp{q}$ Minkowski addition of $\{K_n\},$ denoted by $\Sigma_{q} (\{K_n\}),$ as
    \begin{equation}
        \label{eq:InfiniteLittleLpMinkowskiAddition}
        \Sigma_{q}(\{K_n\}) \coloneq
        \bigcup \left\{\sum_{n=1}^\infty a_n v_n\colon v_n \in K_n \text{ for } n \geq 1, \{a_n\} \in \lp{q'}, \norm{\{a_n\}}{\lp{q'}} \leq 1 \right\}.
    \end{equation}
    The sums in~\eqref{eq:InfiniteLittleLpMinkowskiAddition} are convergent due to
    the summability condition~\eqref{eq:ConvexBodiesLittleLpSummability} and Hölder's inequality.
    In particular, it holds that
    \begin{equation*}
        \left|\Sigma_{q} (\{K_n\})\right| \leq \left(\sum_{n=1}^\infty |K_n|^q\right)^{1/q}
    \end{equation*}
    for any $1 \leq q < \infty$ whenever the right-hand side of this expression is finite.
    Furthermore, given a $(d \times d)$ matrix $A,$ it also holds by linearity that
    \begin{equation*}
        A \Sigma_q (\{K_n\}) = \Sigma_q (\{A K_n\}).
    \end{equation*}
    For $1 \leq q < \infty$ it also happens that $\Sigma_{q} (\{K_n\})$ is a convex symmetric bounded closed set,
    which is proved in next lemma.
    
    \begin{lemma}
        \label{lemma:LittleLqNormsConvexBody}
        Let $1\leq q<\infty$ and let $\{ K_{m}\}^{\infty}_{m=1}$ be a sequence of sets in $\convex$ with
        \begin{equation*}
            \sum_{m=1}^{\infty}|K_m|^{q}<\infty.
        \end{equation*}
        Then $K = \Sigma_{q} (\{K_n\})$
        is a well-defined set in $\convex$ with
        \begin{equation*}
            |K|\sim_{d,q}\left(\sum_{m=1}^{\infty}|K_m|^{q}\right)^{1/q}.
        \end{equation*}
    \end{lemma}

    \begin{proof}
        First of all, for all sequences $\{ v_{m}\}_{m=1}^{\infty}$ with $v_m\in K_m,$ for all $m=1,2,\ldots$ and for all complex numbers $a_1,a_2,\ldots$ with $\norm{\{a_m\}}{q'}\leq1$ we have by H\"older's inequality
        \begin{equation*}
            \sum_{m=1}^{\infty}|a_mv_m|\leq\norm{\{a_m\}}{q'}\left(\sum_{m=1}^{\infty}|v_m|^{q}\right)^{1/q}\leq\left(\sum_{m=1}^{\infty}|K_m|^{q}\right)^{1/q}<\infty.
        \end{equation*}
        This shows that $K$ is a well-defined bounded set with $|K|\leq \left(\sum_{m=1}^{\infty}|K_m|^{q}\right)^{1/q}.$

        It is obvious that $K$ is symmetric. In addition, a direct computation shows that $K$ is convex.


        Let us now show that $K$ is closed. In fact, we will show directly that $K$ is compact. This follows from a standard weak star compactness argument, whose details we present for the reader's convenience.

        Since $\complex^{d}$ is a separable metric space, it suffices to show that $K$ is sequentially compact. Let $\{x_{r}\}^{\infty}_{r=1}$ be a sequence of elements of $K$. Then, for all $r=1,2,\ldots$ we can write
        \begin{equation*}
            x_{r} = \sum_{m=1}^{\infty}a^{r}_{m}v_{m}^{r},
        \end{equation*}
        where $a^{r}=\{a^{r}_m\}_{m=1}^{\infty}\in\ell^{q'}$, $\Vert a^{r}\Vert_{\ell^{q'}}\leq1$, and $v_{m}^{r}\in K_m$, for all $m=1,2,\ldots$.

        It is well-known that $\ell^{q'}$ is isometrically isomorphic to the dual $(\ell^{q})^{*}$ of $\ell^{q}$ under the natural pairing, since $q<\infty$ (see for instance \cite[Chapter III, Examples 5.9 and 5.10]{Conway2007}). Moreover, it is clear that $\ell^{q}$ is a a separable Banach space. Therefore, by the Banach--Alaoglou theorem (\cite[Chapter V, Theorems 3.1 and 5.1]{Conway2007}) we deduce that the closed unit ball of $\ell^{q'}$ is sequentially compact under the weak star topology on $\ell^{q'}$. Clearly, $\{a^{r}\}_{r=1}^{\infty}$ is a sequence in the closed unit ball of $\ell^{q'}$. We deduce that after passing to a subsequence of $\{a^r\}^{\infty}_{r=1}$, one can find $a=\{a_m\}_{m=1}^{\infty}\in\ell^{q'}$, $\Vert a\Vert_{\ell^{q'}}\leq1$ with $a^{r}\to a$ in the weak star topology of $\ell^{q'}$ as $r\to\infty$. Notice that to ease notation, we have not introduced a new notation for this subsequence of $\{a^{r}\}_{r=1}^{\infty}$.

        Notice that
        \begin{equation*}
            \Vert \{|a^{r}_m-a_{m}|\}_{m=1}^{\infty}\Vert_{\ell^{q'}} \leq 2,\quad\forall r=1,2,\ldots.
        \end{equation*}
        Clearly, a dilation of a compact set in a topological vector space remains a compact set. Thus, by the same token as before, after passing to a further subsequence we can find $b\in\ell^{q'}$ with $\Vert b\Vert_{\ell^{q'}}\leq2$ such that
        \begin{equation*}
            \{|a^{r}_m-a_{m}|\}_{m=1}^{\infty}\to b
        \end{equation*}
        in the weak star topology of $\ell^{q'}$ as $r\to\infty$. Since convergence in the weak star topology implies pointwise convergence on $\mathbb{N}$ we deduce
        \begin{equation*}
            a^{r}\to a
        \end{equation*}
        as well as
        \begin{equation*}
            \{|a^{r}_m-a_{m}|\}_{m=1}^{\infty}\to b
        \end{equation*}
        pointwise on $\mathbb{N}$ as $r\to\infty$, showing that $b$ must be the zero sequence.
        
        We now claim that
        \begin{equation}
            \label{eq:step1}
            \sum_{m=1}^{\infty}a^{r}_{m}v_{m}^{r} - \sum_{m=1}^{\infty}a_{m}v_{m}^{r}\to 0
        \end{equation}
        in $\complex^d$ as $r\to\infty$. Indeed, we compute
        \begin{equation*}
            \left|\sum_{m=1}^{\infty}a^{r}_{m}v_{m}^{r} - \sum_{m=1}^{\infty}a_{m}v_{m}^{r}\right|
            \leq\sum_{m=1}^{\infty}|a^{r}_{m}-a_{m}|\cdot|v_{m}^{r}|
            \leq\sum_{m=1}^{\infty}|a^{r}_{m}-a_{m}|\cdot|K_{m}|.
        \end{equation*}
        Observe that $\{|K_{m}|\}_{m=1}^{\infty}$ is a fixed sequence in $\ell^{q}$. Therefore, the weak star convergence of $\{|a^{r}_m-a_{m}|\}_{m=1}^{\infty}$ against the zero sequence in $\ell^{q'}$ proves~\eqref{eq:step1}.

        Now, for each $m=1,2,\ldots$, $K_m$ is a compact metric space, therefore the cartesian product $\prod_{m=1}^{\infty}K_m$ can also be made into a compact metric space. Thus, after passing to a further subsequence, one can find a sequence $\{v_m\}_{m=1}^{\infty}\in \prod_{m=1}^{\infty}K_m$ such that
        \begin{equation*}
            v_{m}^{r}\to v_{m}
        \end{equation*}
        in $K_m$, that is in $\complex^{d}$ as $r\to\infty$.

        Lastly, we show that
        \begin{equation}
            \label{eq:step2}
            \sum_{m=1}^{\infty}a_{m}v_{m}^{r} \to \sum_{m=1}^{\infty}a_{m}v_{m}
        \end{equation}
        in $\complex^{d}$ as $r\to\infty$. This will show that the sequence $\{x_{r}\}^{\infty}_{r=1}$ with which we began has a subsequence that converges to the element $x:=\sum_{m=1}^{\infty}a_{m}v_{m}$ of $K$.

        Consider the counting measure $\mu$ on $\mathbb{N}$. Observe that
        \begin{equation*}
            \{a_{m}v_{m}^{r}\}_{m=1}^{\infty}\to\{a_{m}v_{m}\}_{m=1}^{\infty}
        \end{equation*}
        pointwise on $\mathbb{N}$. Furthermore, notice that
        \begin{equation*}
            |a_{m}v_{m}^{r}|\leq |a_{m}|\cdot|K_{m}|,\quad\forall r=1,2,\ldots
        \end{equation*}
        and that $\{|a_m|\cdot|K_{m}|\}_{m=1}^{\infty}\in L^1(\mathbb{N},\mu)$. Therefore, an application of the Dominated Convergence Theorem yields immediately~\eqref{eq:step2}.
        
        We will now show that
        \begin{equation*}
            |K|\gtrsim_{d,q}\frac{1}{\sqrt{d}}\left(\sum_{m=1}^{\infty}|K_m|^{q}\right)^{1/q}.
        \end{equation*}
        We can pick a sequence $\{ v_m\}_{m=1}^{\infty}$ with $v_m\in K_m$ and $|v_m|\geq\frac{1}{2}|K_m|,$ for all $m=1,2,\ldots.$ It is clear that there exists a positive integer $N$ with
        \begin{equation*}
            \sum_{m=1}^{N}|v_m|^{q}\geq\frac{1}{2}\sum_{m=1}^{\infty}|v_m|^{q}.
        \end{equation*}
        For any vector $x\in\complex^d,$ we denote by $x^{1},\ldots,x^{d}$ its coordinates. Notice that
        \begin{equation*}
            \sum_{m=1}^{N}|v_m|^{q}\sim_{d,q}\sum_{m=1}^{N}\sum_{j=1}^{d}|v_m^{j}|^q=\sum_{j=1}^{d}\sum_{m=1}^{N}|v_m^{j}|^q,
        \end{equation*}
        therefore there exists $j\in\{1,\ldots,d\}$ with
        \begin{equation*}
            \sum_{m=1}^{N}|v_m^{j}|^q\gtrsim_{d,q}\sum_{m=1}^{N}|v_m|^{q}.
        \end{equation*}
        Clearly, one can find complex numbers $a_1,\ldots,a_{N}$ with $\norm{\{a_m\}_{m=1}^N}{q'}\leq1$ such that
        \begin{equation*}
            \sum_{m=1}^{N}a_mv_m^{j}=\left(\sum_{m=1}^{N}|v_m^{j}|^q\right)^{1/q}.
        \end{equation*}
        It follows that
        \begin{align*}
            \left|\sum_{m=1}^{N}a_mv_m\right|&\geq\left|\sum_{m=1}^{N}a_mv_m^{j}\right|=\left(\sum_{m=1}^{N}|v_m^{j}|^q\right)^{1/q}\\
            &\gtrsim_{d,q}\left(\sum_{m=1}^{N}|v_m|^{q}\right)^{1/q}\gtrsim_{q}\left(\sum_{m=1}^{\infty}|K_m|^{q}\right)^{1/q}.
        \end{align*}
        Since $\sum_{m=1}^{N}a_mv_m=\sum_{m=1}^{N}a_mv_m+\sum_{m=N+1}^{\infty}0\cdot 0\in K,$ the proof is complete.
    \end{proof}

    We are now ready to prove Theorem~\ref{thm:VectorValuedEstimates}.
    \begin{proof}[Proof~of~Theorem~\ref{thm:VectorValuedEstimates}]
        Fix first $1 < q < \infty.$
        We construct a new family $\cF_q$ of convex body valued functions as follows.
        For each sequence $\{(F_m,G_m)\}^{\infty}_{m=1} \subseteq \cF$ such that
        \begin{equation}
            \label{eq:ConvexBodySummability}
            \sum_{m=1}^\infty |F_m(x)|^q < \infty\quad \text{ and }\quad \sum_{m=1}^\infty |G_m(x)|^q < \infty
        \end{equation}
        at almost every $x \in \reals^n,$
        we define
        \begin{equation*}
            F(x) = \Sigma_q (\{F_m(x)\}^{\infty}_{m=1}),\qquad G(x) = \Sigma_q (\{G_m(x)\}^{\infty}_{m=1}).
        \end{equation*}
        Observe that by Lemma~\ref{lemma:LittleLqNormsConvexBody}, both $F$ and $G$ are convex body valued functions, $F,G:\reals^n\to\mathcal{K}_{\mathrm{bcs}}$ with
        \begin{equation*}
            |W(x)^{1/p}F(x)| \sim_{d,p} \left(\sum_{m=1}^\infty |W(x)^{1/p}F_m(x)|^q\right)^{1/q}
        \end{equation*}
        and
        \begin{equation*}
            |W(x)^{1/p}G(x)| \sim_{d,p} \left(\sum_{m=1}^\infty |W(x)^{1/p}G_m(x)|^q\right)^{1/q}
        \end{equation*}
        at almost every $x \in \reals^n,$ for any $1<p<\infty.$
        
        Let us prove that the convex body valued functions $F,G$ are measurable. For each $m=1,2,\ldots$ define the convex body valued function $F^{m}:\reals^n\to\mathcal{K}_{\mathrm{bcs}}$ by
        \begin{equation*}
            F^{m}(x) := \Sigma_{q}(\{F_t(x)\}^{m}_{t=1}),\quad x\in\reals^{n}.
        \end{equation*}
        Fix $m$. For each $t=1,\ldots,m$, by using the characterization \eqref{eq:SelectionFunctionMeasurability} of measurability of convex body valued functions we can find a sequence $\{f_k^{t}\}_{k=1}^{\infty}$ of measurable selections functions of $F_t$ such that
        \begin{equation*}
            F_{t}(x) = \overline{\{f_k^{t}(x):~k=1,2,\ldots\}}.
        \end{equation*}
        Then, it is easy to see that
        \begin{equation*}
            F^{m}(x) = \overline{\left\{\sum_{t=1}^{m}a_t f^t_{k_{t}}(x)\colon k_t\in\mathbb{N},~a_t\in\mathbb{Q}\text{ for } t=1,\ldots,m\text{ with }\sum_{t=1}^{m}|a_t|^{q'}\leq1\right\}}.
        \end{equation*}
        Thus, by the characterization \eqref{eq:SelectionFunctionMeasurability} of measurability of convex body valued functions coupled with the fact that $\mathbb{N}^{m}$ is still a countable set, it follows that each $F^{m}$ is measurable.
        
        Moreover, it is clear that
        \begin{equation*}
            F(x) = \overline{\bigcup_{m=1}^{\infty}F_m(x)}
        \end{equation*}
        and therefore also
        \begin{equation*}
            F(x) = \overline{\mathrm{conv}\left(\bigcup_{m=1}^{\infty}F_{m}(x)\right)},
        \end{equation*}
        since $F(x)$ is already known to be a closed convex set. Thus by \cite[Theorem 8.2.2]{Aubin_2009} we deduce that $F$ is measurable. Similarly, $G$ is also measurable.

        Next, note that the family of extrapolation pairs $\cF_q$ that we have just constructed
        satisfies inequalities~\eqref{eq:BootstrapInequalities} with exponent $q$ for any matrix weight $W \in A_q.$
        Indeed, for any given matrix weight $W \in A_q,$ due to Lemma~\ref{lemma:LittleLqNormsConvexBody} it holds that
        \begin{equation*}
            \begin{split}
                \norm{F}{\Lp{q}(W)}^q &= \int_{\reals^n} |W^{1/q}F(x)|^q\, \mathdm(x)\\
                &= \int_{\reals^n} \left|\Sigma_q (\{W(x)^{1/q} F_n(x)\})\right|^q\, \mathdm(x)\\
                &\leq C(d,q) \sum_{n=1}^\infty \int_{\reals^n} |W(x)^{1/q} F_n(x)|^q\, \mathdm(x)\\
                &\leq C_{q}([W]_q)^q \sum_{n=1}^\infty \int_{\reals^n} |W(x)^{1/q} G_n(x)|^q\, \mathdm(x)\\
                &\leq C(d,q) C_{q}([W]_q)^q \int_{\reals^n} \left|\Sigma_q (\{W(x)^{1/q} G_n(x)\})\right|^q\, \mathdm(x)\\
                &= C_{q}([W]_q)^q \norm{G}{\Lp{q}(W)}^q.
            \end{split}
        \end{equation*}
        Here we have used Theorem~\ref{thm:Extrapolation} with the extrapolation pairs $(F_n,G_n) \in \cF$
        and with exponent $q,$ so that we have
        \begin{equation*}
            C_q([W]_{A_q}) = C(d,q,p_0) C_{p_0}\left(C(n,d,q,p_0) [W]_{A_q}^{\max\left\{1,\frac{p_0-1}{q-1}\right\}}\right).
        \end{equation*}
        We have just seen that the family of extrapolation pairs $\cF_q$ satisfies~\eqref{eq:BootstrapInequalities},
        therefore we can apply Theorem~\ref{thm:Extrapolation} and Lemma~\ref{lemma:LittleLqNormsConvexBody}
        to get that for any $1 < p < \infty$ and any matrix weight $W \in A_p$ it holds that
        \begin{equation*}
            \begin{split}
                \norm{\left(\sum_{n=1}^\infty |W(x)^{1/p}F_n(x)|^q\right)^{1/q}}{\Lp{p}}
                &\leq C(d,p) \norm{W^{1/p}F}{\Lp{p}}\\
                \leq  C_{p,q}(&[W]_{A_p}) \norm{W^{1/p}G}{\Lp{p}}\\
                \leq C_{p,q}(&[W]_{A_p}) \norm{\left(\sum_{n=1}^\infty |W(x)^{1/p}G_n(x)|^q\right)^{1/q}}{\Lp{p}}
            \end{split}
        \end{equation*}
        for every sequence $\{(F_n,G_n)\} \subseteq \cF$ satisfying~\eqref{eq:ConvexBodySummability}
        and where
        \begin{equation*}
            \begin{split}
                C_{p,q}([W]_{A_p}) &= C(d,p) C_{p,q}\left(C(n,d,p,q) [W]_{A_p}^{\max\left\{1,\frac{q-1}{p-1}\right\}}\right)\\
                &= C(d,p,q) C_{p_0}\left(C(n,d,p_0,p,q) [W]_{A_p}^{\left(\max\left\{1,\frac{p_0-1}{q-1}\right\}\max\left\{1,\frac{q-1}{p-1}\right\}\right)}\right)
            \end{split}
        \end{equation*}
        if $p_0<\infty,$ and
        \begin{equation*}
            \begin{split}
                C_{p,q}([W]_{A_p}) &= C(d,p) C_{p,q}\left(C(n,d,p,q) [W]_{A_p}^{\max\left\{1,\frac{q-1}{p-1}\right\}}\right)\\
                &= C(d,p,q) C_{p_0}\left(C(n,d,p_0,p,q) [W]_{A_p}^{\left(\frac{1}{q-1}\max\left\{1,\frac{q-1}{p-1}\right\}\right)}\right)
            \end{split}
        \end{equation*}
        if $p_0=\infty,$ with the additional factor $C(d,p,q)$ being in both cases due to the various applications of Lemma~\ref{lemma:LittleLqNormsConvexBody}.
    \end{proof}

    \subsection{Fefferman--Stein vector valued inequalities for weighted maximal operators}
    \label{subsec:MatrixWeightedFeffermanStein}
    Next, we give an application of Theorem~\ref{thm:VectorValuedEstimates}.
    Fix $1 < p < \infty.$
    Given a $(d \times d)$ matrix valued weight $W$ on $\reals^n,$
    one can define the pointwise matrix weighted maximal operator for vector valued functions by
    \begin{equation*}
        M_W f(x) \coloneq \sup_Q \frac{1}{|Q|} \int_Q |W(x)^{1/p}f(y)|\, \mathdm(y) \1_Q(x),
    \end{equation*}
    where $f$ is a locally integrable function taking values on $\complex^d$
    and where the supremum is taken over all cubes with sides parallel to the coordinate axes.
    The maximal operator $M_W$ is called Christ--Goldberg maximal operator,
    since it was introduced and studied by these two authors (see~\cite{Christ_Goldberg_2001} and~\cite{Goldberg_2003}).
    One can also define a weighted maximal operator with reducing operators as
    \begin{equation*}
        \widetilde{M}_W f(x) \coloneq \sup_Q \int_Q |\cW_Q f(y)|\, \mathdm(y) \1_Q(x),
    \end{equation*}
    where $f$ is a locally integrable function with values on $\complex^d,$
    the supremum is taken over bounded cubes with sides parallel to the axes
    and where $\cW_Q$ is the reducing operator of $W$ over $Q$ with exponent $p.$
    We do not make the dependence on the exponent $p$ explicit since it will always be clear from the context.
    
    Both weighted maximal operators play an important role in the theory of matrix weighted norm inequalities.
    In addition, whenever $W$ is a matrix $A_p$ weight, both maximal operators are bounded
    from weighted $\Lp{p}(W)$ to unweighted $\Lp{p}.$
    Isralowitz and Moen~\cite{Isralowitz_Moen_2019} proved that
    \begin{equation}
        \label{eq:PWMaximalOperatorBound}
        \norm{M_W}{\Lp{p}(W)\rightarrow\Lp{p}} \lesssim_{n,d,p} [W]_{A_p}^{1/(p-1)},
    \end{equation}
    while Isralowitz, Kwon and Pott~\cite{Isralowitz_Kwon_Pott2017} showed that
    \begin{equation*}
        \norm{\widetilde{M}_W}{\Lp{p}(W)\rightarrow\Lp{p}} \lesssim_{n,d,p} [W]_{A_p}^{1/(p-1)}.
    \end{equation*}

    We will use a trick due to Bownik and Cruz-Uribe~\cite{Cruz_Uribe_Bownik_Extrapolation} so as to use
    Theorem~\ref{thm:VectorValuedEstimates} (which is a statement about convex body valued functions)
    to show statements about $\complex^d$ vector valued functions.
    In order to apply the previous results for convex body valued operators to these weighted maximal operators,
    we will also need to define the convex body valued analogue of $M_W.$
    Both definitions will be related by the following correspondence between vector valued and convex body valued functions.
    Given a vector valued function $f,$ that is taking values on $\complex^d,$
    we define the convex body valued function $F$ by
    \begin{equation}
        \label{eq:VectorToConvexBody}
        F(x) \coloneq \{\lambda f(x)\colon |\lambda| \leq 1\}.
    \end{equation}
    By construction, $F$ takes values on $\convex(\complex^d)$ and also $|F(x)| = |f(x)|.$
    Thus, any estimate for operators depending on $|f(x)|$ can be studied through
    estimates for an analogous operator acting on $F(x)$ and depending as well on $|F(x)|$ uniquely.
    It is straightforward to check that, for a locally integrable vector valued function $f$ and a cube $Q,$
    it holds that the averaging operator $A_Q$ applied to $F$ given by~\eqref{eq:VectorToConvexBody} can be computed as
    \begin{equation}
        \label{eq:VectorAveragingOperator}
        A_Q F(x) = \left\{\frac{1}{|Q|} \int_Q k(y)f(y)\, \mathdm(y)\colon k \in \Lp{\infty}, \norm{k}{\Lp{\infty}} \leq 1\right\} \1_Q(x).
    \end{equation}
    For this reason, we will abuse notation and denote $A_Q f(x) = A_Q F(x).$
    Also note that, by an appropriate choice of the function $k$ in~\eqref{eq:VectorAveragingOperator}, one can see that
    \begin{equation}
        \label{eq:AveragingModulusEquivalence}
        |A_Q f(x)| \sim_d \frac{1}{|Q|} \int_Q |f(y)|\, \mathdm(y),
    \end{equation}
    and more generally
    \begin{equation*}
        |A_Q F(x)| \sim_d \frac{1}{|Q|} \int_Q |F(y)|\, \mathdm(y)
    \end{equation*}
    for convex body valued functions in general, which can be seen by choosing an appropriate selection function.
    For fixed $1 < p < \infty$ and given a $(d \times d)$-matrix weight $W,$
    the convex body valued analogue of $M_W$ acting on a locally integrably bounded convex body valued function $F$
    is defined by
    \begin{equation*}
        M_W^\cK F(x) \coloneq \overline{\mathrm{conv}\left(\bigcup_Q A_Q(W(x)^{1/p}F(x))\right)},
    \end{equation*}
    where the union is taken over all bounded cubes with sides parallel to the axes.
    Moreover, given a locally integrable vector valued function $f,$ we denote $M_W^\cK f(x) = M_W^\cK F(x),$
    where $F$ is the convex body valued function given by~\eqref{eq:VectorToConvexBody}.
    In addition, as proved in \cite[Lemma 3.4]{Vuorinen2024}, if $W$ is a matrix weight and $f$ a vector valued function,
    linearity and~\eqref{eq:AveragingModulusEquivalence} yield that
    \begin{equation}
        \label{eq:MaximalOperatorsEquivalences}
        |M_W^\cK f(x)| = |W(x)^{1/p} M^\cK f(x)| \sim_d M_W f(x),
    \end{equation}
    where $M^\cK$ denotes the maximal operator weighted by the $(d \times d)$-identity matrix.

    \theoremstyle{plain}
    \newtheorem*{ThmMaximalPointwise}{Theorem~\ref{thm:FeffermanSteinPointwiseWeightedMaximal}}
    \begin{ThmMaximalPointwise}
        Consider a $(d \times d)$ matrix weight $W$ and a sequence of vector valued functions $\{f_n\}.$
        Then, for each $1 < p,q < \infty$ it holds that
        \begin{equation*}
            \norm{\left(\sum_{n=1}^\infty |M_W f_n|^q\right)^{1/q}}{\Lp{p}}
            \leq C(n,d,p,q,[W]_{A_p}) \norm{\left(\sum_{n=1}^\infty |W(x)^{1/p} f_n|^q\right)^{1/q}}{\Lp{p}},
        \end{equation*}
        where
        \begin{equation*}
            C(n,d,p,q,[W]_{A_p}) = C(n,d,p,q) [W]_{A_p}^{\max\left\{\frac{1}{q-1},\frac{1}{p-1}\right\}}.
        \end{equation*}
    \end{ThmMaximalPointwise}
    \begin{proof}
        In order to apply Theorem~\ref{thm:VectorValuedEstimates},
        we need to construct a family $\cF$ of extrapolation pairs for which
        the left-hand side of~\eqref{eq:BootstrapInequalities} is finite.
        To this end, we will restrict ourselves to vector valued functions $f \in \Lp{\infty}_\mathrm{c}(\reals^n; \complex^d)$
        (compactly supported essentially bounded functions),
        and a density argument will yield the conclusion for $f \in \Lp{p}(\reals^n; \complex^d).$
        For each vector valued function $f \in \Lp{\infty}_\mathrm{c},$ we consider the extrapolation pair $(M^\cK F,F),$
        where $F$ is the convex body valued function defined by~\eqref{eq:VectorToConvexBody}.

        Consider the given $1 < q < \infty.$
        It is clear that the left-hand side of~\eqref{eq:BootstrapInequalities} is finite for every pair $(M^\cK F,F) \in \cF$ for $p_0 = q,$
        and that~\eqref{eq:BootstrapInequalities} is satisfied if we take $C_{p_0}(t) = C_q(t) = C(n,d,q) t^{\frac{1}{q-1}}.$
        Indeed, because of~\eqref{eq:MaximalOperatorsEquivalences}, for any matrix $A_q$ weight $W$ we have that
        \begin{equation*}
                \norm{M^\cK F}{\Lp{q}(W)} \sim_d \norm{M_W f}{\Lp{q}}
                \leq C(n,d,q) [W]_{A_q}^{1/(q-1)} \norm{f}{\Lp{q}(W)} < \infty,
        \end{equation*}
        where we have used estimate~\eqref{eq:PWMaximalOperatorBound} and
        that $\Lp{\infty}_\mathrm{c}(\reals^n; \complex^d) \subseteq \Lp{q}(\reals^n; \complex^d).$

        Now, Theorem~\ref{thm:VectorValuedEstimates} gives that
        \begin{equation}
            \label{eq:FeffermanSteinPWMOConvexBodies}
            \norm{\left(\sum_{n=1}^\infty |W^{1/p} M^\cK F_n|^{q_1}\right)^{1/{q_1}}}{\Lp{p}}
            \leq C(n,d,p_0,p,q_1,[W]_{A_p}) \norm{\left(\sum_{n=1}^\infty |W^{1/p} F_n|^{q_1}\right)^{1/{q_1}}}{\Lp{p}}
        \end{equation}
        for any $1 < p,q_1 < \infty$ (with $q_1$ possibly different if $q = p_0$) and any matrix weight $W \in A_p,$
        where $\{(M^\cK F_n, F_n)\}$ is a sequence in $\cF$ and where
        \begin{align*}
            &C(n,d,p_0,p,q_1,[W]_{A_p})\\
            &= C(d,p) C_{p_0}\left(C(n,d,p_0,p,q_1) [W]_{A_p}^{\left(\max\left\{1,\frac{p_0-1}{q_1-1}\right\}\max\left\{1,\frac{q_1-1}{p-1}\right\}\right)}\right).
        \end{align*}
        In particular, if we restrict our attention to $q_1 = p_0 = q,$ we can take
        \begin{equation*}
            C(n,d,q,p,[W]_{A_p}) = C(n,d,q,p) [W]_{A_p}^{\max\left\{\frac{1}{q-1},\frac{1}{p-1}\right\}}.
        \end{equation*}
        Also observe that for each sequence $\{(M^\cK F_n, F_n)\} \subseteq \cF,$
        we can choose a sequence $\{f_n\} \subseteq \Lp{\infty}_\mathrm{c} (\reals^n; \complex^d)$ such that
        $F_n$ is obtained from $f_n$ using~\eqref{eq:VectorToConvexBody} for every $n \geq 1.$
        On one hand observe that using that $|W^{1/p} F_n| = |W^{1/p} f_n|$
        for the sequence $\{f_n\} \subseteq \Lp{\infty}_\mathrm{c} (\reals^n; \complex^d)$ that we chose previously,
        we get that
        \begin{equation}
            \label{eq:RightConvexToVectorEquivalence}
            \norm{\left(\sum_{n=1}^\infty |W^{1/p} F_n|^q\right)^{1/q}}{\Lp{p}}
            = \norm{\left(\sum_{n=1}^\infty |W^{1/p} f_n|^q\right)^{1/q}}{\Lp{p}}.
        \end{equation}
        On the other hand, using~\eqref{eq:MaximalOperatorsEquivalences}, we see that
        \begin{equation}
            \label{eq:LeftConvexToVectorEquivalences}
            \norm{\left(\sum_{n=1}^\infty |W^{1/p} M^\cK F_n|^q\right)^{1/q}}{\Lp{p}}
            \sim_{d} \norm{\left(\sum_{n=1}^\infty |M_W f_n|^q\right)^{1/q}}{\Lp{p}}.
        \end{equation}
        Finally, replacing~\eqref{eq:RightConvexToVectorEquivalence} and~\eqref{eq:LeftConvexToVectorEquivalences}
        into~\eqref{eq:FeffermanSteinPWMOConvexBodies}, we get the desired vector valued estimates
        for the Christ--Goldberg maximal operator, as we wanted to show.
    \end{proof}
    
    The analogous result for the operator $\widetilde{M}_W$ is a consequence of Theorem~\ref{thm:FeffermanSteinPointwiseWeightedMaximal}.
    Note here that the exponent that we get for $[W]_{A_p}$ in this case
    is larger than that of Theorem~\ref{thm:FeffermanSteinPointwiseWeightedMaximal}.
    This is different from what is known in the scalar context,
    where the sharp exponent for the Christ--Goldberg maximal auxiliary operator
    is the same as for the Christ--Goldberg maximal operator
    (see~\cite[Lemma~3.5]{Isralowitz_Moen_2019}).
    The reason for this larger exponent in our case is that first we bound
    $\widetilde{M}_W f$ by the Hardy--Littlewood maximal function $M(M_W f)$
    of the Christ--Goldberg maximal operator,
    and then we apply Theorem~\ref{thm:FeffermanSteinPointwiseWeightedMaximal}.
    In the scalar setting, the proof of the analogous result relies on
    the reverse Hölder inequality for $A_p$ weights.
    We do not know if the exponent appearing in
    Theorem~\ref{thm:FeffermanSteinReducingWeightedMaximal} can be improved.

    \begin{theorem}
        \label{thm:FeffermanSteinReducingWeightedMaximal}
        Consider a $(d \times d)$ matrix weight $W$ and a sequence of vector valued functions $\{f_n\}.$
        Then, for each $1 < p,q < \infty$ it holds that
        \begin{equation*}
            \norm{\left(\sum_{n=1}^\infty |\widetilde{M}_W f_n|^q\right)^{1/q}}{\Lp{p}}
            \leq C(n,d,p,q,[W]_{A_p}) \norm{\left(\sum_{n=1}^\infty |W(x)^{1/p} f_n|^q\right)^{1/q}}{\Lp{p}},
        \end{equation*}
        where
        \begin{equation*}
            C(n,d,p,q,[W]_{A_p}) = C(n,d,p,q) [W]_{A_p}^{\frac{1}{p} + \max\left\{\frac{1}{q-1},\frac{1}{p-1}\right\}}.
        \end{equation*}
    \end{theorem}
    \begin{proof}
        Just observe that
        \begin{equation*}
            \begin{split}
                \widetilde{M}_W f(x) &= \sup_{Q\ni x} \frac{1}{|Q|} \int_Q |\cW_Q f(y)|\, \mathdm(y)\\
                &\sim_d \sup_{Q\ni x} \frac{1}{|Q|} \int_Q \left( \frac{1}{|Q|} \int_Q |W(z)^{1/p} f(y)|^p\, \mathdm(z)\right)^{1/p}\, \mathdm(y)\\
                &\leq C(d,p)[W]_{A_p}^{1/p} \sup_{Q\ni x} \frac{1}{|Q|} \int_Q \frac{1}{|Q|} \int_Q |W(z)^{1/p} f(y)|\, \mathdm(z)\, \mathdm(y)\\
                &\leq C(d,p)[W]_{A_p}^{1/p} \sup_{Q\ni x} \frac{1}{|Q|} \int_Q M_W f(z)\, \mathdm(z)\\
                &= C(d,p)[W]_{A_p}^{1/p} M(M_W f) (x).
            \end{split}
        \end{equation*}
        Therefore, we get that
        \begin{equation*}
                \norm{\left(\sum_{n=1}^\infty |\widetilde{M}_W f_n|^q\right)^{1/q}}{\Lp{p}}
            \leq C(d,p) [W]_{A_p}^{1/p} \norm{\left(\sum_{n=1}^\infty |M(M_W f_n)|^q\right)^{1/q}}{\Lp{p}}.
        \end{equation*}
        An application of the classical Fefferman--Stein vector valued inequalities for the maximal operator
        followed by the use of Theorem~\ref{thm:FeffermanSteinPointwiseWeightedMaximal}
        yields the desired result.
    \end{proof}

    \section{Two matrix weighted biparameter product BMO}
    \label{s:MatrixWeighedProductBMO}

    In this second main part of the present work dyadic grids of cubes in $\reals^n$, product dyadic grids in $\reals^n\times\reals^m$ and the associated Haar systems play a fundamental role. See Section~\ref{s:background} for the notation.

    Fix a product dyadic grid $\bfD=\cD^1\times\cD^2$ in $\reals^n\times\reals^m.$ Let $1<p<\infty,$ and let $U,V$ be biparameter $(d\times d)$ matrix $\bfD$-dyadic $A_p$ weights on $\reals^n\times\reals^m.$

    Denote by $B$ any sequence $\{B_{R}^{\e}\}_{\substack{R\in\bfD\\\e\in\cE}}$ in $\mathrm{M}_{d}(\complex).$ We emphasize that $\cE$ stands for the set of biparameter signatures, $\cE=\cE^1\times\cE^2,$ where $\cE^1=\{0,1\}^{n}\setminus\{(1,\ldots,1)\}$ and $\cE^2=\{0,1\}^m\setminus\{(1,\ldots,1)\}.$ We define
    \begin{equation*}
        \Vert B\Vert_{\BMOprodD(U,V,p)}\coloneq \sup_{\Omega}\frac{1}{|\Omega|^{1/2}}\bigg(\sum_{\substack{R\in\bfD(\Omega)\\\e\in\cE}}|\cV_{R}B_{R}^{\e}\cU_{R}^{-1}|^2\bigg)^{1/2},
    \end{equation*}
    where the supremum ranges over all Lebesgue-measurable subsets $\Omega$ of $\reals^{n+m}$ of nonzero finite measure, and all reducing operators are taken with respect to exponent $p.$ Note that
    \begin{equation*}
        |\cV_{R}P\cU^{-1}_{R}|=|\cU^{-1}_{R}P^{\ast}\cV_{R}|\sim_{p,d}|\cU'_{R}P^{\ast}(\cV'_{R})^{-1}|,\qquad\forall R\in\bfD,~\forall P\in \mathrm{M}_{d}(\complex).
    \end{equation*}
    Therefore
    \begin{equation*}
        \Vert B\Vert_{\BMOprodD(U,V,p)}\sim_{p,d}\Vert B^{\ast}\Vert_{\BMOprodD(V',U',p')}.
    \end{equation*}

    \subsection{\texorpdfstring{$\mathrm{H}^1$-BMO duality}{H1-BMO duality}} The main goal of this subsection is to prove Theorem \ref{thm:H1BMOduality}. We split the proof in Propositions~\ref{prop:BMOtofunctional} and~\ref{prop:functionaltoBMO} below, each treating one of the bounds in the equivalence stated in Theorem~\ref{thm:H1BMOduality}.

    We define $\mathrm{H}^1_{\bfD}(U,V,p)$ as the set of all sequences $\Phi=\{\Phi_{R}^{\e}\}_{\substack{R\in\bfD\\\e\in\cE}}$ in $\mathrm{M}_{d}(\complex)$ such that
    \begin{equation*}
        \Vert\Phi\Vert_{\mathrm{H}^1_{\bfD}(U,V,p)}\coloneq \bigg\Vert\bigg(\sum_{\substack{R\in\bfD\\\e\in\cE}}|V^{-1/p}\Phi_{R}^{\e}\cU_{R}|^2\frac{\1_{R}}{|R|}\bigg)^{1/2}\bigg\Vert_{\Lp{1}(\reals^{n+m})}<\infty.
    \end{equation*}
    This is the direct biparameter analog of the one-parameter two matrix weighted $\mathrm{H}^1$ norm defined in \cite{Isralowitz_2017}. It is easy to check that $(\mathrm{H}^1_{\bfD}(U,V,p),\Vert\var\Vert_{\mathrm{H}^1_{\bfD}(U,V,p)})$ is a Banach space.

    \begin{proposition}
        \label{prop:BMOtofunctional}
    Let $B\in\emph{BMO}_{\emph{prod},\bfD}(U,V,p).$ Then, the linear functional $\ell_{B}:\mathrm{H}^1_{\bfD}(U,V,p)\rightarrow\complex$ given by
    \begin{equation*}
        \ell_{B}(\Phi)\coloneq \sum_{\substack{R\in\bfD\\\e\in\cE}}\mathrm{tr}((B_{R}^{\e})^{\ast}\Phi_{R}^{\e}),\quad\Phi\in \mathrm{H}^1_{\bfD}(U,V,p)
    \end{equation*}
    is well-defined and bounded with $\Vert\ell_{B}\Vert\lesssim_{p,d}[V]_{A_p,\bfD}^{2/p}\Vert B\Vert_{\emph{BMO}_{\emph{prod},\bfD}(U,V,p)}.$
    \end{proposition}

    \begin{proof}
        The proof of this result is essentially the same as in
        the one-parameter case for $p = 2,$
        which corresponds to the first half of the proof of \cite[Theorem 1.3]{Isralowitz_2017}.
        The necessary tools for the argument have been developed previously in other papers.
        We include the full proof for the reader's convenience,
        indicating were to find the necessary facts to complete the arguments
        in this more general context.
        
        Let $\Phi\in \mathrm{H}^1_{\bfD}(U,V,p)$ be arbitrary. We show that the sum
        \begin{equation*}
            \ell_{B}(\Phi)\coloneq \sum_{\substack{R\in\bfD\\\e\in\cE}}\tr((B_{R}^{\e})^{\ast}\Phi_{R}^{\e})
        \end{equation*}
        converges absolutely with
        \begin{equation*}
            |\ell_{B}(\Phi)|\lesssim_{p,d}[V]_{A_p,\bfD}^{2/p}\Vert B\Vert_{\BMOprodD(U,V,p)}\Vert\Phi\Vert_{\mathrm{H}^1_{\bfD}(U,V,p)}.
        \end{equation*}
        We have
        \begin{align*}
            &\sum_{\substack{R\in\bfD\\\e\in\cE}}|\tr((B_{R}^{\e})^{\ast}\Phi_{R}^{\e})|
            =\sum_{\substack{R\in\bfD\\\e\in\cE}}|\tr(\cV_{R}^{-1}\Phi_{R}^{\e}\cU_{R}\cU_{R}^{-1}(B_{R}^{\e})^{\ast}\cV_{R})|\\
            &\lesssim_{d}\sum_{\substack{R\in\bfD\\\e\in\cE}}|\cV_{R}^{-1}\Phi_{R}^{\e}\cU_{R}\cU_{R}^{-1}(B_{R}^{\e})^{\ast}\cV_{R}|
            \leq\sum_{\substack{R\in\bfD\\\e\in\cE}}|\cV_{R}^{-1}\Phi_{R}^{\e}\cU_{R}|\cdot|\cU_{R}^{-1}(B_{R}^{\e})^{\ast}\cV_{R}|\\
            &\lesssim_{d}\sum_{\substack{R\in\bfD\\\e\in\cE}}|\cV'_{R}\Phi_{R}^{\e}\cU_{R}|\cdot|\cV_{R}B_{R}^{\e}\cU_{R}^{-1}|.
        \end{align*}
        The passage from the second line to the third one can be done using that
        $|\cV^{-1}_R e| \leq  |\cV'_R e|$ for any $e \in \complex^d.$
        In particular, if this fact is applied to the matrix norm $|\cV_{R}^{-1}\Phi_{R}^{\e}\cU_{R}|,$ one gets the stated inequality with an implicit constant depending on $d.$
        Set now
        \begin{equation*}
            F\coloneq \bigg(\sum_{\substack{R\in\bfD\\\e\in\cE}}|V^{-1/p}\Phi_{R}^{\e}\cU_{R}|^2\frac{\1_{R}}{|R|}\bigg)^{1/2}
        \end{equation*}
        and define
        \begin{equation*}
            \Omega_k\coloneq \{ F>2^k\},\qquad k\in\integer,
        \end{equation*}
        \begin{equation*}
            \cB_k\coloneq \left\{ R\in\bfD:~|R\cap\Omega_{k+1}|\leq\frac{1}{2}|R|<|R\cap\Omega_{k}|\right\},\qquad k\in\integer,
        \end{equation*}
        \begin{equation*}
            \widetilde{\Omega}_k\coloneq \left\{ M_{\bfD}(\1_{\Omega_k})>\frac{1}{2}\right\},\qquad k\in\integer.
        \end{equation*}
        Clearly $\Omega_k\subseteq\widetilde{\Omega}_k$ up to a set of zero measure, and in fact $|\widetilde{\Omega}_k|\sim|\Omega_k|,$ for all $k\in\integer.$ It is also obvious that $R\subseteq\widetilde{\Omega}_k,$ for all $R\in\cB_k,$ for all $k\in\integer.$        
        Using the fact that $F\in\Lp{1}(\reals^{n+m}),$ it is easy to see that that for every $R\in\bfD$ with $\Phi_{R}^{\e}\neq0$ for some $\e\in\cE$ there exists $k\in\integer$ such that $R\in\cB_k.$ Thus, we have
        \begin{align*}
            |\ell_{B}(\Phi)|
            &\lesssim_{p,d} \sum_{k\in\integer}\sum_{\substack{R\in\cB_k\\\e\in\cE}}|\cV'_{R}\Phi_{R}^{\e}\cU_{R}|\cdot|\cV_{R}B_{R}^{\e}\cU_{R}^{-1}|\\
            &\leq\sum_{k\in\integer}\bigg(\sum_{\substack{R\in\cB_k\\\e\in\cE}}|\cV'_{R}\Phi_{R}^{\e}\cU_{R}|^2\bigg)^{1/2}\bigg(\sum_{\substack{R\in\cB_k\\\e\in\cE}}|\cV_{R}B_{R}^{\e}\cU_{R}^{-1}|^2\bigg)^{1/2}\\
            &\leq\Vert B\Vert_{\BMOprodD(U,V,p)}\sum_{k\in\integer}\bigg(\sum_{\substack{R\in\cB_k\\\e\in\cE}}|\cV'_{R}\Phi_{R}^{\e}\cU_{R}|^2\bigg)^{1/2}|\text{sh}(\cB_k)|^{1/2}\\
            &\lesssim\Vert B\Vert_{\BMOprodD(U,V,p)}\sum_{k\in\integer}\bigg(\sum_{\substack{R\in\cB_k\\\e\in\cE}}|\cV'_{R}\Phi_{R}^{\e}\cU_{R}|^2\bigg)^{1/2}|\widetilde{\Omega}_k|^{1/2}.
        \end{align*}
        We will now prove that
        \begin{equation}
            \label{goal}
            \sum_{\substack{R\in\cB_k\\\e\in\cE}}|\cV'_{R}\Phi_{R}^{\e}\cU_{R}|^2\lesssim_{p,d}[V]_{A_p,\bfD}^{4/p}2^{2k}|\widetilde{\Omega}_{k}|,\qquad\forall k\in\integer.
        \end{equation}
        This will be enough to conclude the proof, because assuming it we will get that the sum over $k$ of the terms in the left-hand side of~\eqref{goal} will be bounded by
        \begin{equation*}
            [V]_{A_p,\bfD}^{2/p}\sum_{k\in\integer}2^{k}|\widetilde{\Omega}_{k}|
            \sim[V]_{A_p,\bfD}^{2/p}\Vert F\Vert_{\Lp{1}(\reals^{n+m})}.
        \end{equation*}

        Fix now $k\in\integer.$ We begin by noticing that
        \begin{equation*}
            \int_{\widetilde{\Omega}_k\setminus\Omega_{k+1}}F(x)^2\mathdm(x)\leq 2^{2k+2}|\widetilde{\Omega}_k\setminus\Omega_{k+1}|\leq 2^{2k+2}|\widetilde{\Omega}_k|,
        \end{equation*}
        by the definition of $\Omega_{k+1}.$ Moreover, denoting by $e_1,\ldots,e_d$ the standard basis vectors in $\complex^{d},$ we have
        \begin{align}
            \label{FSquaredLower}
            &\int_{\widetilde{\Omega}_k\setminus\Omega_{k+1}}F(x)^2\mathdm(x)\geq
            \int_{\widetilde{\Omega}_k\setminus\Omega_{k+1}}\sum_{\substack{R\in\cB_k\\\e\in\cE}}|V(x)^{-1/p}\Phi_{R}^{\e}\cU_{R}|^2\frac{\1_{R}(x)}{|R|}\mathdm(x)\\
            &=\sum_{\substack{R\in\cB_k\\\e\in\cE}}\frac{1}{|R|}\int_{R\setminus\Omega_{k+1}}|V(x)^{-1/p}\Phi_{R}^{\e}\cU_{R}|^2\mathdm(x)\\
            &\sim_{d}\sum_{j=1}^{d}\sum_{\substack{R\in\cB_k\\\e\in\cE}}\frac{1}{|R|}\int_{R\setminus\Omega_{k+1}}|V(x)^{-1/p}\Phi_{R}^{\e}\cU_{R}e_{j}|^2\mathdm(x)\\
            &\geq\sum_{j=1}^{d}\sum_{\substack{R\in\cB_k\\\e\in\cE}}\frac{|R\setminus\Omega_{k+1}|}{|R|}\left(\frac{1}{|R\setminus\Omega_{k+1}|}\int_{R\setminus\Omega_{k+1}}|V(x)^{-1/p}\Phi_{R}^{\e}\cU_{R}e_{j}|\mathdm(x)\right)^2\\
            &\sim\sum_{j=1}^{d}\sum_{\substack{R\in\cB_k\\\e\in\cE}}\left(\frac{1}{|R\setminus\Omega_{k+1}|}\int_{R\setminus\Omega_{k+1}}|V(x)^{-1/p}\Phi_{R}^{\e}\cU_{R}e_{j}|\mathdm(x)\right)^2\\
            &\sim\sum_{j=1}^{d}\sum_{\substack{R\in\cB_k\\\e\in\cE}}\left(\frac{1}{|R|}\int_{R\setminus\Omega_{k+1}}|V(x)^{-1/p}\Phi_{R}^{\e}\cU_{R}e_{j}|\mathdm(x)\right)^2.
        \end{align}
        Let $v\in\complex^d$ be arbitrary. Consider the function
        \begin{equation*}
            w\coloneq |V^{-1/p}v|^{p'}.
        \end{equation*}
        Then, $w$ is a scalar $\bfD$-dyadic biparameter $A_{p'}$ weight on $\reals^n\times\reals^m$ with
        \begin{equation*}
            [w]_{A_{p'},\bfD}\lesssim_{p,d}[V']_{A_{p'},\bfD}\sim_{p,d}[V]_{A_p,\bfD}^{p'/p},
        \end{equation*}
        because $V^{-1/p}$ is a $(d\times d)$ matrix $\bfD$-dyadic biparameter $A_p$ weight (see for example Lemma 3.2 in \cite{DKPS2024}). It is then well-known that
        \begin{equation*}
            [w^{1/p'}]_{A_2,\bfD}\leq[w]_{A_{p'},\bfD}^{1/p'}
        \end{equation*}
        and
        \begin{equation*}
            \langle w\rangle_{R}^{1/p'}\leq[w]_{A_{p'},\bfD}^{1/p'}\langle w^{1/p'}\rangle_{R},\quad\forall R\in\bfD
        \end{equation*}
        (see for example Subsection 2.3.3 in \cite{Kakaroumpas_Soler_2022}). Using Jensen's inequality and the definition of the $A_2$ characteristic, it follows that
        \begin{equation*}
            \frac{w^{1/p'}(R\setminus\Omega_{k+1})}{w^{1/p'}(R)}\geq[w^{1/p'}]_{A_2,\bfD}^{-1}\frac{|R\setminus\Omega_{k+1}|}{|R|}\geq[w^{1/p'}]_{A_2,\bfD}^{-1}\cdot\frac{1}{2},
        \end{equation*}
        so
        \begin{equation*}
            \int_{R\setminus\Omega_{k+1}}|V(x)^{-1/p}v|\mathdm(x)\gtrsim_{p,d} [V]_{A_p,\bfD}^{-1/p}\int_{R}|V(x)^{-1/p}v|\mathdm(x).
        \end{equation*}
        Thus, the last part of~\eqref{FSquaredLower} has the lower bound
        \begin{align*}
            &[V]_{A_p,\bfD}^{-2/p}\sum_{j=1}^{d}\sum_{\substack{R\in\cB_k\\\e\in\cE}}\left(\frac{1}{|R|}\int_{R}|V(x)^{-1/p}\Phi_{R}^{\e}\cU_{R}e_{j}|\mathdm(x)\right)^2\\
            &\gtrsim_{p,d}[V]_{A_p,\bfD}^{-4/p}\sum_{j=1}^{d}\sum_{\substack{R\in\cB_k\\\e\in\cE}}\left(\frac{1}{|R|}\int_{R}|V(x)^{-1/p}\Phi_{R}^{\e}\cU_{R}e_{j}|^{p'}\mathdm(x)\right)^{2/p'}\\
            &\sim_{p,d}[V]_{A_p,\bfD}^{-4/p}\sum_{j=1}^{d}\sum_{\substack{R\in\cB_k\\\e\in\cE}}|\cV'_{R}\Phi_{R}^{\e}\cU_{R}e_{j}|^2\sim_{d}[V]_{A_p,\bfD}^{-4/p}\sum_{\substack{R\in\cB_k\\\e\in\cE}}|\cV'_{R}\Phi_{R}^{\e}\cU_{R}|^2,
        \end{align*}
        concluding the proof.
    \end{proof}

    Before we proceed to the second half of Theorem~\ref{thm:H1BMOduality},
    recall that the \emph{dyadic biparameter Christ--Goldberg maximal operator} corresponding to a weight $W$ on $\reals^{n+m}$ (and exponent $p$) is defined as
        \begin{equation*}
            M_{\bfD,W}f(x) \coloneq \sup_{R\in\bfD}\langle|W(x)^{1/p}f|\rangle_{R}\1_{R}(x),\qquad x\in\reals^{n+m},~f\in \mathrm{L}^1_{\mathrm{loc}}(\reals^{n+m};\complex^d),
        \end{equation*}
    and the \emph{dyadic biparameter Christ--Goldberg auxiliary maximal operator} as
        \begin{equation*}
             \widetilde{M}_{\bfD,W}f\coloneq \sup_{R\in\bfD}\langle|\cW_{R}f|\rangle_{R}\1_{R},\qquad x\in\reals^{n+m},~f\in \mathrm{L}^1_{\mathrm{loc}}(\reals^{n+m};\complex^d).
        \end{equation*}
    Because of~\cite[Theorem~1.3]{Vuorinen2024} we have that the operator $M_{\bfD,W}$ is bounded when acting on $\Lp{p}(W) \rightarrow \Lp{p}(\reals^{n+m})$
    (note that the target space is unweighted) provided that $[W]_{A_p,\bfD}<\infty,$ specifically one has the bound
    \begin{equation}
        \label{strong dyadic Christ-Goldberg}
        \Vert M_{\bfD,W}\Vert_{\Lp{p}(W)\rightarrow \Lp{p}(\reals^{n+m})}\lesssim_{n,m,p,d}[W]_{A_p,\bfD}^{2/(p-1)},
    \end{equation}
    for $1<p<\infty.$
    From \cite[Proposition 4.1]{DKPS2024} we also have
        \begin{equation*}
            \Vert \widetilde{M}_{\bfD,U}\Vert_{\Lp{p}(U)\rightarrow \Lp{p}(\reals^{n+m})}\lesssim_{n,m,p,d}[W]_{A_p,\bfD}^{(p+1)/(p(p-1))}.
        \end{equation*}

    \begin{proposition}
        \label{prop:functionaltoBMO}
        Let $\ell$ be any bounded linear functional on $\mathrm{H}^1_{\bfD}(U,V,p).$ Then, there exists a unique $B\in\emph{BMO}_{\emph{prod},\bfD}(U,V,p)$ with
        \begin{equation}
            \label{eq:FunctionalExpression}
            \ell(\Phi)=\sum_{\substack{R\in\bfD\\\e\in\cE}}\mathrm{tr}((B_{R}^{\e})^{\ast}\Phi_{R}^{\e}),\quad\forall\Phi\in \mathrm{H}^1_{\bfD}(U,V,p).
        \end{equation}
        Moreover, there holds
        \begin{equation*}
            \Vert B\Vert_{\emph{BMO}_{\emph{prod},\bfD}(U,V,p)}\lesssim_{n,m,p,d}[V]_{A_{p},\bfD}^{2+1/p}\Vert\ell\Vert.
        \end{equation*}
    \end{proposition}

    \begin{proof}
        We partially adapt the second half of the proof of \cite[Theorem 1.3]{Isralowitz_2017}.
        The main difference is that we need to introduce an auxiliary function
        $N_\Omega$ and the properties of scalar weights of the form $|V^{1/p}A|^{p}$
        with positive definite matrices $A$ in order to get the necessary bounds. There is also an extra twist given in \eqref{eq:BMOp} below.
        
        We denote by $\cH$ the set of all $\mathrm{M}_{d}(\complex)$-valued $\Lp{2}$ functions on $\reals^{n+m}$ with finitely many nonzero biparameter cancellative Haar coefficients, and we identify such functions with finitely supported sequences $\Phi=\{\Phi_{R}^{\e}\}_{\substack{R\in\bfD\e\in\cE}}$ in $\mathrm{M}_{d}(\complex)$ in the obvious way.
        In addition, given a matrix $A$ we denote here its $(i,j)$ entry by $A(i,j).$
        Then, for all $\Phi\in\cH,$ we have
        \begin{equation*}
            \Phi=\sum_{\substack{R\in\bfD\\\e\in\cE}}h_{R}^{\e}\Phi_{R}^{\e}
        \end{equation*}
        and therefore
        \begin{equation}
            \label{PhiBPairing}
            \ell(\Phi)=\sum_{\substack{R\in\bfD\\\e\in\cE}}\ell(h_{R}^{\e}\Phi_{R}^{\e})=\sum_{\substack{R\in\bfD\\\e\in\cE}}\sum_{i,j=1}^{d}\ell(h_{R}^{\e}\Phi_{R}^{\e}(i,j)E_{ij})
            =\sum_{\substack{R\in\bfD\\\e\in\cE}}\sum_{i,j=1}^{d}h_{R}^{\e}\Phi_{R}^{\e}(i,j)\ell(E_{ij})
        \end{equation}
        where $E_{ij}$ is the $d\times d$-matrix with 1 at the $(i,j)$-entry and 0 at all other entries.
        Thus, if we take $B=\{B_{R}^{\e}\}_{\substack{R\in\bfD\\\e\in\cE}}$ given by
        \begin{equation*}
            B_{R}^{\e}(i,j)\coloneq \overline{\ell(h_{R}^{\e}E_{ji})},\qquad\forall i,j=1,\ldots,d,~\forall R\in\bfD,~\forall\e\in\cE,
        \end{equation*}
        the last expression of~\eqref{PhiBPairing} is exactly~\eqref{eq:FunctionalExpression}. 
        By Proposition \ref{prop:BMOtofunctional} and since $\cH$ is dense in $\mathrm{H}^1_{\bfD}(U,V,p),$ it suffices only to prove that sequence $B$ is in $\BMOprodD(U,V,p)$ with
        \begin{equation}
            \label{BMOestimate}
            \Vert B\Vert_{\BMOprodD(U,V,p)}\lesssim_{n,m,p,d}[V]_{A_{p},\bfD}^{2}\Vert\ell\Vert.
        \end{equation}

        Note that by the scalar, unweighted BMO equivalences we have
        \begin{align}
            \label{eq:BMOp}
            \nonumber&\Vert B\Vert_{\BMOprodD(U,V,p)}\\
            &\sim_{n,m,p,d}\sup_{\Omega}\frac{1}{|\Omega|^{1/p'}}\bigg\Vert\bigg(\sum_{\substack{R\in\bfD(\Omega)\\\e\in\cE}}|\cV_{R}B_{R}^{\e}\cU_{R}^{-1}|^2\frac{\1_{R}}{|R|}\bigg)^{1/2}\bigg\Vert_{\Lp{p'}(\reals^{n+m})}=:\cC,
        \end{align}
        so it suffices only to prove that $\cC\lesssim_{n,m,p,d}\Vert\ell\Vert.$ By the Monotone Convergence Theorem, we can without loss of generality assume that $B$ has only finitely many nonzero terms.

        Let us denote by $\langle\var,\var\rangle$ the Hermitian product on $\mathrm{M}_{d}(\complex)$ given by
        \begin{equation*}
            \langle A,B\rangle\coloneq \tr(B^{\ast}A).
        \end{equation*}
        Then, it is elementary to see that
        \begin{align}
            \label{Riesz representation}
            \nonumber&\Vert F\Vert_{\Lp{p}(\reals^{n+m};\mathrm{M}_{d}(\complex))}\\
            \nonumber&\sim_{n,m,p,d}\left\{\left|\int_{\reals^{n+m}}\langle F(x),G(x)\rangle\mathdm(x)\right|:~ G\in\cH,~\Vert G\Vert_{\Lp{p'}(\reals^{n+m};\mathrm{M}_{d}(\complex))}=1\right\}\\
            &=\bigg\{\bigg|\sum_{\substack{R\in\bfD\\\e\in\cE}}\tr((G^{\e}_{R})^{\ast}F_{R}^{\e})\bigg|:~ G\in\cH,~\Vert G\Vert_{\Lp{p'}(\reals^{n+m};\mathrm{M}_{d}(\complex))}=1\bigg\},
        \end{align}
        for all $F\in\cH.$ We recall also the usual (unweighted) dyadic Littlewood--Paley estimates
        \begin{align}
            \label{Littlewood-Paley}
            \nonumber&\Vert F\Vert_{\Lp{p}(\reals^{n+m};\mathrm{M}_{d}(\complex))}\\
            &\sim_{n,m,p,d}\bigg(\int_{\reals^{n+m}}\bigg(\sum_{\substack{R\in\bfD\\\e\in\cE}}|F_{R}^{\e}|^2\frac{\1_{R}(x)}{|R|}\bigg)^{p/2}dx\bigg)^{1/p},\qquad\forall F\in\cH.
        \end{align}
        Fix now any Lebesgue-measurable subset $\Omega$ of $\reals^{n+m}$ with finite nonzero measure. Then, we have
        \begin{align*}
            &\bigg(\int_{\Omega}\bigg(\sum_{\substack{R\in\bfD(\Omega)\\\e\in\cE}}|\cV_{R}B_{R}^{\e}\cU_{R}^{-1}|^2\frac{\1_{R}(x)}{|R|}\bigg)^{p'/2}\mathdm(x)\bigg)^{1/p'}\\
            &\overset{\eqref{Littlewood-Paley}}{\sim}_{n,m,p,d}\bigg\Vert\sum_{\substack{R\in\bfD(\Omega)\\\e\in\cE}}h_{R}^{\e}\cV_{R}B_{R}^{\e}\cU_{R}^{-1}\bigg\Vert_{\Lp{p'}(\reals^{n+m};\mathrm{M}_{d}(\complex))}\\
            &\overset{\eqref{Riesz representation}}{\sim}_{n,m,p,d}\sup_{A\in\cH\setminus\{0\}}
            \frac{1}{\Vert A\Vert_{\Lp{p}(\reals^{n+m};\mathrm{M}_{d}(\complex))}}\bigg|\sum_{\substack{R\in\bfD(\Omega)\\\e\in\cE}}\tr\left((A_{R}^{\e})^{\ast}\cV_{R}B_{R}^{\e}\cU_{R}^{-1}\right)\bigg|\\
            &=\sup_{A\in\cH\setminus\{0\}}
            \frac{1}{\Vert A\Vert_{\Lp{p}(\reals^{n+m};\mathrm{M}_{d}(\complex))}}\bigg|\sum_{\substack{R\in\bfD(\Omega)\\\e\in\cE}}\tr\left((B_{R}^{\e})^{\ast}\cV_{R}A_{R}^{\e}\cU_{R}^{-1}\right)\bigg|\\
            &=\sup_{A\in\cH\setminus\{0\}}\frac{1}{\Vert A\Vert_{\Lp{p}(\reals^2;\mathrm{M}_{d}(\complex))}}|(\hat{A},B)|=\sup_{A\in\cH\setminus\{0\}}\frac{1}{\Vert A\Vert_{\Lp{p}(\reals^2;\mathrm{M}_{d}(\complex))}}|\ell(\hat{A})|\\
            &\leq\Vert\ell\Vert\sup_{A\in\cH\setminus\{0\}}\frac{1}{\Vert A\Vert_{\Lp{p}(\reals^2;\mathrm{M}_{d}(\complex))}}\Vert\hat{A}\Vert_{\mathrm{H}^1_{\bfD}(U,V,p)},
        \end{align*}
        where
        \begin{equation*}
            \hat{A}^{\e}_{R}\coloneq 
            \begin{cases}
                \cV_{R}A_{R}^{\e}\cU_{R}^{-1},\text{ if } R\in\bfD(\Omega)\\\\
                0,\text{ otherwise}
            \end{cases}
            ,\quad R\in\bfD,~\e\in\cE.
        \end{equation*}
        It suffices now to prove that
        \begin{equation*}
            \Vert\hat{A}\Vert_{\mathrm{H}^1_{\bfD}(U,V,p)}\lesssim_{n,m,p,d}[V]_{A_{p},\bfD}^{2}|\Omega|^{1/p'}\Vert A\Vert_{\Lp{p}(\reals^2;\mathrm{M}_{d}(\complex))},
        \end{equation*}
        for all $A\in\cH.$ We set
        \begin{equation*}
            N_{\Omega}(x)\coloneq \sup_{R\in\bfD(\Omega)}|V(x)^{-1/p}\cV_{R}|\1_{R}(x),\quad x\in\reals^{n+m}.
        \end{equation*}
        We define now functions $\widetilde{N}_\Omega^{(k)}$ by applying
        the dyadic Christ--Goldberg biparameter maximal operator $M_{\bfD,V'}$ with respect to the biparameter $(d\times d)$ matrix valued $\bfD$-dyadic $A_{p'}$ weight $V'$ on the function $V^{1/p}\1_{\Omega}e_k$ for each $k=1,\ldots,d$ (where we recall that $\{e_1,\ldots,e_d\}$ is the standard basis of $\complex^d$).
        That is, we define
        \begin{equation*}
            \widetilde{N}_\Omega^{(k)}(x) \coloneq \sup_{R\in\bfD} \langle|V(x)^{-1/p}V^{1/p}\1_\Omega e_k|\rangle_R \1_R(x).
        \end{equation*}
        If we consider $\widetilde{N}_\Omega \coloneq \sup_{k=1,\ldots,d} \widetilde{N}_\Omega^{(k)},$ this satisfies
        \begin{equation*}
            \widetilde{N}_\Omega(x) \gtrsim_{n,m,d,p} [V]_{A_p,\bfD}^{-1/p} N_\Omega(x).
        \end{equation*}
        Indeed, just observe that
        \begin{align*}
            N_{\Omega}(x)&=\sup_{R\in\bfD(\Omega)}|V(x)^{-1/p}\cV_{R}|\1_{R}(x)=\sup_{R}|\cV_{R}V(x)^{-1/p}|\1_{R}(x)\\
            &\sim_{d}\sup_{R\in\bfD(\Omega)}\langle|V^{1/p}V(x)^{-1/p}|^{p}\rangle^{1/p}\1_{R}(x)\\
            &\lesssim_{n,m,d,p}[V]_{A_p,\bfD}^{1/p} \sup_{R\in\bfD(\Omega)} \langle|V^{1/p}V(x)^{-1/p}|\rangle_R \1_R(x)\\
            &=[V]_{A_p,\bfD}^{1/p} \sup_{R\in\bfD(\Omega)} \langle|V(x)^{-1/p}V^{1/p}|\rangle_R \1_R(x),
        \end{align*}
        because the scalar weight $|V^{1/p}M|^{p}$ is uniformly in the biparameter $\bfD$-dyadic Muckenhoupt $A_p$ class for every positive definite matrix $M$ (see for example \cite[Lemma 3.4]{DKPS2024}).
        Next, using the comparability between the matrix norm and the supremum of norms of matrix columns, we get that the previous expression is bounded above by
        \begin{equation*}
            [V]_{A_p,\bfD}^{1/p}
            \sup_{R\in\bfD(\Omega)} \sup_{k=1,\ldots,d} \langle|V(x)^{-1/p}V^{1/p} e_k|\rangle_R \1_R(x)
            \leq [V]_{A_p,\bfD}^{1/p}\widetilde{N}_{\Omega}(x).
        \end{equation*}
        Finally, using the boundedness \eqref{strong dyadic Christ-Goldberg} of the dyadic Christ--Goldberg biparameter maximal operator,
        we obtain
        \begin{equation*}
            \Vert N_{\Omega}\Vert_{\Lp{p'}(\reals^2)}\lesssim_{n,m,p,d}[V]_{A_p,\bfD}^{1/p}[V']_{A_{p'},\bfD}^{2/(p'-1)}|\Omega|^{1/p'}\sim_{p,d}[V]_{A_{p},\bfD}^{2+1/p}|\Omega|^{1/p'}.
        \end{equation*}
        Thus, for all $A\in\cH$ we have
        \begin{align*}
            \Vert\hat{A}\Vert_{\mathrm{H}^1_{\bfD}(U,V,p)}&=\int_{\Omega}\bigg(\sum_{\substack{R\in\bfD(\Omega)\\\e\in\cE}}|V(x)^{-1/p}\cV_{R}A_{R}^{\e}\cU_{R}^{-1}\cU_{R}|^2\frac{\1_{R}(x)}{|R|}\bigg)^{1/2}\mathdm(x)\\
            &=\int_{\Omega}\bigg(\sum_{\substack{R\in\bfD(\Omega)\\\e\in\cE}}|V(x)^{-1/p}\cV_{R}A_{R}^{\e}|^2\frac{\1_{R}(x)}{|R|}\bigg)^{1/2}\mathdm(x)\\
            &\leq\int_{\Omega}N_{\Omega}(x)\bigg(\sum_{\substack{R\in\bfD(\Omega)\\\e\in\cE}}|A_{R}^{\e}|^2\frac{\1_{R}(x)}{|R|}\bigg)^{1/2}\mathdm(x)\\
            &\lesssim_{n,m,p,d}\Vert N_{\Omega}\Vert_{\Lp{p'}(\reals^2)}\Vert A\Vert_{\Lp{p}(\reals^{n+m};\mathrm{M}_{d}(\complex))}\\
            &\lesssim_{n,m,p,d}[V]_{A_{p},\bfD}^{2+1/p}|\Omega|^{1/p'}\Vert A\Vert_{\Lp{p}(\reals^{n+m};\mathrm{M}_{d}(\complex))},
        \end{align*}
        concluding the proof of \eqref{BMOestimate}.

        Uniqueness of $B$ follows immediately by testing $\ell$ on sequences in $\cH.$
    \end{proof}

    \begin{remark}
    The proofs of Propositions \ref{prop:BMOtofunctional}, \ref{prop:functionaltoBMO} work also in the one-parameter setting.
    \end{remark}



    \subsection{Two-matrix weighted bounds for paraproducts}
    For a given locally integrable function $B:\reals^{n+m}\to \mathrm{M}_{d}(\complex),$ we define its $\Vert B\Vert_{\BMOprodD(U,V,p)}$ norm
    as the norm of the sequence $\{B_{R}^{\e}\}$ of its biparameter Haar coefficients. Following the terminology of Holmes--Petermichl--Wick \cite[Subsection 6.1]{Holmes_Petermichl_Wick_2018} we define the following so-called ``pure'' biparameter paraproducts acting on (suitable) $\complex^d$-valued functions $f$ on $\reals^{n+m}$:
    \begin{equation*}
        \Pi_{\bfD,B}^{(11)}f\coloneq \sum_{\substack{R\in\bfD\\\e\in\cE}}h^{\e}_{R}B_{R}^{\e}\langle f\rangle_{R},\quad \Pi_{\bfD,B}^{(00)}f\coloneq \sum_{\substack{R\in\bfD\\\e\in\cE}}\frac{\1_{R}}{|R|}B^{\e}_{R}f^{\e}_{R},
    \end{equation*}
    \begin{equation*}
        \Gamma_{\bfD,B}f\coloneq \sum_{\substack{R\in\bfD\\\e,\delta\in\cE\\\e_i\neq\delta_i,~i=1,2}}\frac{1}{\sqrt{|R|}}h_{R}^{1\oplus\e\oplus\delta}B_{R}^{\e}f_{R}^{\delta}.
    \end{equation*}
    Here we denote
    \begin{equation*}
        1\oplus1=0\oplus0=0,\quad 1\oplus0=0\oplus1=1
    \end{equation*}
    and extend these operations in the obvious entrywise fashion to $\cE^1,$ $\cE^2$ and $\cE.$ Moreover, abusing notation we denote $(1,\ldots,1)$ (where the number of entries is always clear from the context) by just $1.$ Notice that $1\oplus\e_i\oplus\delta_i\neq 1,$ for all $\e,\delta\in\cE$ with $\e_i\neq\delta_i,$ $i=1,2.$

    Clearly $(\Pi_{\bfD,B}^{(00)})^{\ast}=\Pi_{\bfD,B^{\ast}}^{(11)}$ in the unweighted $\Lp{2}(\reals^{n+m};\complex^d)$ sense. Observe also that a change of summation variables yields
    \begin{equation*}
        \Gamma_{\bfD,B}f=\sum_{\substack{R\in\bfD\\\e,\delta\in\cE\\\e_i\neq\delta_i,~i=1,2}}\frac{1}{\sqrt{|R|}}h_{R}^{\e}B_{R}^{1\oplus\e\oplus\delta}f_{R}^{\delta},
    \end{equation*}
    therefore $(\Gamma_{\bfD,B})^{\ast}=\Gamma_{\bfD,B^{\ast}}$ in the unweighted $\Lp{2}(\reals^{n+m};\complex^d)$ sense.

    \begin{proposition}
        \label{prop:pureparaproducts}
        Let $d,p,U,V$ and $B$ be as above.
        \begin{enumerate}
            \item[(a)] There holds
            \begin{equation*}
                \Vert\Pi_{\bfD,B}^{(11)}\Vert_{\Lp{p}(U)\rightarrow \Lp{p}(V)}\sim\Vert\Pi_{\bfD,B}^{(00)}f\Vert_{\Lp{p}(U)\rightarrow \Lp{p}(V)}\sim\Vert B\Vert_{\emph{BMO}_{\emph{prod},\bfD}(U,V,p)},
            \end{equation*}
    where all implied constants depend only on $n,m,d,p,[U]_{A_p,\bfD}$ and $[V]_{A_p,\bfD}.$

            \item[(b)] There holds
            \begin{equation*}
                \Vert\Gamma_{\bfD,B}\Vert_{\Lp{p}(U)\rightarrow \Lp{p}(V)}\lesssim\Vert B\Vert_{\emph{BMO}_{\emph{prod},\bfD}(U,V,p)},
            \end{equation*}
        where all implied constants depend only on $n,m,d,p,[U]_{A_p,\bfD}$ and $[V]_{A_p,\bfD}.$
        \end{enumerate}
    \end{proposition}

    \begin{proof}
        Throughout the proof $\lesssim,\gtrsim,\sim$ mean that all implied inequality constants depend only on $n,m,p,d,[U]_{A_p,\bfD}$ and $[V]_{A_p,\bfD},$ unless otherwise specified.

        \item[(a)] We adapt part of the proof of \cite[Theorem 2.2]{Isralowitz_2017}. Note that by the John--Nirenberg inequalities for (unweighted) scalar dyadic product BMO we have
        \begin{align*}
            &\Vert B\Vert_{\BMOprodD(U,V,p)}\\
            &\sim_{p,d,n,m}\sup_{\Omega}\frac{1}{|\Omega|^{1/p}}\bigg\Vert\bigg(\sum_{\substack{R\in\bfD(\Omega)\\\e\in\cE}}|\cV_{R}B_{R}^{\e}\cU_{R}^{-1}|^2\frac{\1_{R}}{|R|}\bigg)^{1/2}\bigg\Vert_{\Lp{p}(\reals^{n+m})}\coloneq \cC,
        \end{align*}
        where the supremum is taken over all Lebesgue-measurable subsets $\Omega$ of $\reals^2$ of nonzero finite measure. Therefore, it suffices to prove that
        \begin{equation*}
            \Vert\Pi_{\bfD,B}^{(11)}\Vert_{\Lp{p}(U)\rightarrow \Lp{p}(\reals^{n+m};\complex^d)}\gtrsim\cC
        \end{equation*}
        and
        \begin{equation*}
            \Vert\Pi_{\bfD,B}^{(11)}\Vert_{\Lp{p}(U)\rightarrow \Lp{p}(\reals^{n+m};\complex^d)}\lesssim\Vert B\Vert_{\BMOprodD(U,V,p)}.
        \end{equation*}

        Let us first see that $\cC\lesssim\Vert\Pi_{\bfD,B}^{(11)}\Vert_{\Lp{p}(U)\rightarrow \Lp{p}(\reals^{n+m};\complex^d)}.$ Let $\Omega$ be any Lebesgue-measurable subset of $\reals^2$ of nonzero finite measure. Let also $e\in\complex^d\setminus\{0\}$ be arbitrary. We test $\Pi_{\bfD,B}^{(11)}$ on the function $f\coloneq \1_{\Omega}U^{-1/p}e.$ Using~\cite[Lemma~5.3]{DKPS2024} we obtain
        \begin{align*}
            \Vert\Pi_{\bfD,B}^{(11)}f\Vert_{\Lp{p}(\reals^{n+m};\complex^d)}
            &\gtrsim\bigg(\int_{\reals^{n+m}}\bigg(\sum_{\substack{R\in\bfD\\\e\in\cE}}|\cV_{R}B_{R}^{\e}\langle f\rangle_{R}|^2\frac{\1_{R}(x)}{|R|}\bigg)^{p/2}\mathdm(x)\bigg)^{1/p}\\
            &\geq\bigg(\int_{\reals^{n+m}}\bigg(\sum_{\substack{R\in\bfD(\Omega)\\\e\in\cE}}|\cV_{R}B_{R}^{\e}\langle U^{-1/p}\rangle_{R}e|^2\frac{\1_{R}(x)}{|R|}\bigg)^{p/2}\mathdm(x)\bigg)^{1/p}.
        \end{align*}
        Note also that $\Vert f\Vert_{\Lp{p}(U)}=|\Omega|^{1/p}\cdot|e|.$ Therefore, the previous quantity is bounded above by
        $\Vert\Pi_{\bfD,B}^{(11)}\Vert_{\Lp{p}(U)\rightarrow \Lp{p}(\reals^{n+m};\complex^d)} |\Omega|^{1/p} |e|$ for all $e\in\complex^d.$
        Using repeatedly the elementary estimates (3.1) and (3.3) in \cite{DKPS2024} we conclude
        \begin{align*}
            &\frac{1}{|\Omega|^{1/p}}\bigg(\int_{\reals^{n+m}}\bigg(\sum_{\substack{R\in\bfD(\Omega)\\\e\in\cE}}|\cV_{R}B_{R}^{\e}\langle U^{-1/p}\rangle_{R}|^2\frac{\1_{R}(x)}{|R|}\bigg)^{p/2}\mathdm(x)\bigg)^{1/p}\\
            &\lesssim\Vert\Pi_{\bfD,B}^{(11)}\Vert_{\Lp{p}(U)\rightarrow \Lp{p}(\reals^{n+m};\complex^d)}.
        \end{align*}
        To continue the computation, we need the following fact proved in \cite[Lemma 2.2]{Isralowitz_Kwon_Pott2017}: If $\cW_{E}$ is the reducing operator of a matrix weight $W$ over a set $E$ with respect to exponent $p,$ and $W^{-1/(p-1)}$ is also integrable over $E$ with reducing exponent $\cW'_{E}$ with respect to exponent $p'$, then
        \begin{equation}
            \label{replace inverse by prime_2}
            |{\langle W^{1/p}\rangle}_{E} e|\leq
            |\cW_{E} e|\lesssim_{p,d}|\cW_{E}\cW_{E}'|^{d}{\langle W^{1/p}\rangle}_{E} e|,\quad\forall e\in\complex^{d}.
        \end{equation}
        Using \eqref{replace inverse by prime_2} we have
        \begin{equation*}
            |P\langle U^{-1/p}\rangle_{R}|=|\langle U^{-1/p}\rangle_{R}P^{\ast}|\sim|\cU'_{R}P^{*}|\sim|\cU^{-1}_{R}P^{*}|=|P\cU^{-1}_{R}|,\quad\forall P\in \mathrm{M}_{d}(\complex),
        \end{equation*}
        yielding the desired result.

        We now prove that $\Vert\Pi_{\bfD,B}^{(11)}\Vert_{\Lp{p}(U)\rightarrow \Lp{p}(\reals^{n+m};\complex^d)}\lesssim\Vert B\Vert_{\BMOprodD(U,V,p)}.$ Using \cite[Lemma~5.3]{DKPS2024} we have
        \begin{align*}
            \Vert\Pi_{\bfD,B}^{(11)}f\Vert_{\Lp{p}(\reals^{n+m};\complex^d)}&\sim\bigg(\int_{\reals^{n+m}}\bigg(\sum_{\substack{R\in\bfD\\\e\in\cE}}|\cV_{R}B_{R}^{\e}\langle f\rangle_{R}|^2\frac{\1_{R}(x)}{|R|}\bigg)^{p/2}\mathdm(x)\bigg)^{1/p}\\
            &\leq\bigg(\int_{\reals^{n+m}}\bigg(\sum_{\substack{R\in\bfD\\\e\in\cE}}|\cV_{R}B_{R}^{\e}\cU_{R}^{-1}|^2|\cU_{R}\langle f\rangle_{R}|^2\frac{\1_{R}(x)}{|R|}\bigg)^{p/2}\mathdm(x)\bigg)^{1/p}\\
            &\leq\bigg(\int_{\reals^{n+m}}\bigg(\sum_{R\in\bfD}|\cV_{R}B_{R}^{\e}\cU_{R}^{-1}|^2(\langle\widetilde{M}_{\bfD,U}f\rangle_{R})^2\frac{\1_{R}(x)}{|R|}\bigg)^{p/2}\mathdm(x)\bigg)^{1/p}\\
            &\sim_{p,n,m}\Vert\Pi_{\bfD,b}^{(11)}(\widetilde{M}_{\bfD,U}f)\Vert_{\Lp{p}(\reals^{n+m})},
        \end{align*}
        where $b=(b_{R}^{\e})_{\substack{R\in\bfD(\Omega)\\\e\in\cE}}$ is the sequence given by
        \begin{equation*}
            b_{R}^{\e}\coloneq |\cV_{R}B_{R}^{\e}\cU_{R}^{-1}|,\quad R\in\bfD,~\e\in\cE.
        \end{equation*}
        By the well-known unweighted bounds for paraproducts in the scalar setting (see e.~g.~\cite{Blasco_Pott_2005}) we have
        \begin{equation*}
            \Vert\Pi_{\bfD,b}^{(11)}\Vert_{\Lp{p}(\reals^{n+m})\rightarrow \Lp{p}(\reals^{n+m})}\lesssim_{p,n,m}\Vert b\Vert_{\BMOprodD}.
        \end{equation*}
        By definition, $\Vert b\Vert_{\BMOprodD}=\Vert B\Vert_{\BMOprodD(U,V,p)}.$
        Putting these facts together and the bound for the dyadic auxiliary biparameter maximal operator, we get
        \begin{equation*}
            \Vert\Pi_{\bfD,B}^{(11)}f\Vert_{\Lp{p}(\reals^{n+m};\complex^d)} \lesssim \Vert B\Vert_{\BMOprodD(U,V,p)}\Vert f\Vert_{\Lp{p}(U)}.
        \end{equation*}

        Finally, by duality we obtain
        \begin{equation*}
            \Vert \Pi_{\bfD,B}^{(00)}\Vert_{\Lp{p}(U)\rightarrow \Lp{p}(V)}=\Vert\Pi_{\bfD,B^{\ast}}^{(11)}\Vert_{\Lp{p'}(V')\rightarrow \Lp{p'}(U')}\sim_{p,d}\Vert B\Vert_{\BMOprodD(U,V,p)}.
        \end{equation*}

        \item[(b)] We adapt the factorization trick from the proof of \cite[Proposition 6.1]{Holmes_Petermichl_Wick_2018}. We have
        \begin{align*}
            (\Gamma_{\bfD,B}f,g)=\sum_{\substack{R\in\bfD\\\e,\delta\in\cE\\\e_i\neq\delta_i,~i=1,2}}\langle B_{R}^{1\oplus\e\oplus\delta}f_{R}^{\e},g_{R}^{\delta}\rangle.
        \end{align*}
        Observe that if $A\in \mathrm{M}_{d}(\complex)$ and $x,y\in\complex^d,$ then
        \begin{align*}
            \langle Ax,y\rangle&=x^{T}A^{T}\bar{y}=\tr(x^{T}A^{T}\bar{y})=\overline{\tr(A^{\ast}y\bar{x}^{T})}.
        \end{align*}
        It follows that
        \begin{align*}
            |(\Gamma_{\bfD,B}f,g)|&
            =\bigg|\sum_{\substack{R\in\bfD\\\e,\delta\in\cE\\\e_i\neq\delta_i,~i=1,2}}\frac{1}{\sqrt{|R|}}\text{tr}((B_{R}^{\e})^{\ast}g_{R}^{1\oplus\e\oplus\delta}\overline{f_{R}^{\delta}}^{T})\bigg|\\
            &\lesssim\Vert B\Vert_{\BMOprodD(U,V,p)}\Vert\Phi\Vert_{\mathrm{H}^1_{\bfD}(U,V,p)},
        \end{align*}
        where
        \begin{equation*}
            \Phi_{R}^{\e}\coloneq \frac{1}{\sqrt{|R|}}\sum_{\substack{\delta\in\cE\\\delta_i\neq\e_i,~i=1,2}}
            g_{R}^{1\oplus\e\oplus\delta}\overline{f_{R}^{\delta}}^{T},\quad R\in\bfD,~\e\in\cE.
        \end{equation*}
        It suffices now to prove that
        \begin{equation*}
            \Vert\Phi\Vert_{\mathrm{H}^1_{\bfD}(U,V,p)}\lesssim\Vert f\Vert_{\Lp{p}(U)}\Vert g\Vert_{\Lp{p'}(V')}.
        \end{equation*}
        We have that $\Vert\Phi\Vert_{\mathrm{H}^1_{\bfD}(U,V,p)}$ is bounded by
        \begin{align*}
            & 2^{n+m}\sum_{\substack{R\in\bfD\\\e,\delta\in\cE\\\e_i\neq\delta_i,~i=1,2}}|V(x)^{-1/p} g_{R}^{1\oplus\e\oplus\delta}\overline{f_{R}^{\delta}}^{T}\cU_{R}|^2\frac{\1_{R}(x)}{|R|^2}\\
            &\leq 2^{n+m}\sum_{\substack{R\in\bfD\\\e,\delta\in\cE\\\e_i\neq\delta_i,~i=1,2}}|V(x)^{-1/p} g_{R}^{1\oplus\e\oplus\delta}|^2\cdot|\overline{f_{R}^{\delta}}^{T}\cU_{R}|^2\frac{\1_{R}(x)}{|R|^2}\\
            &=2^{n+m}\sum_{\substack{R\in\bfD\\\e,\delta\in\cE\\\e_i\neq\delta_i,~i=1,2}}|V(x)^{-1/p} g_{R}^{\e}|^2\frac{\1_{R}(x)}{|R|}\cdot|\cU_{R}f_{R}^{\delta}|^2\frac{\1_{R}(x)}{|R|}\\
            &\leq2^{n+m}S_{\bfD,V'}g(x)\widetilde{S}_{\bfD,U}f(x),
        \end{align*}
        for all $x\in\reals^{n+m},$ where
        \begin{equation*}
            S_{\bfD,V'}g\coloneq \bigg(\sum_{\substack{R\in\bfD\\\e\in\cE}}|V(x)^{-1/p}g_{R}^{\e}|^2\frac{\1_{R}}{|R|}\bigg)^{1/2},\quad\widetilde{S}_{\bfD,U}f\coloneq \bigg(\sum_{\substack{R\in\bfD\\\delta\in\cE}}|\cU_{R}f_{R}^{\delta}|^2\frac{\1_{R}}{|R|}\bigg)^{1/2}.
        \end{equation*}
        Therefore
        \begin{equation*}
            \Vert\Phi\Vert_{\mathrm{H}^1_{\bfD}(U,V,p)}\leq 2^{n+m}\Vert S_{\bfD,V'}g\Vert_{\Lp{p'}(\reals^{n+m})}\Vert\widetilde{S}_{\bfD,U}f\Vert_{\Lp{p}(\reals^{n+m})}.
        \end{equation*}
        It is proved in \cite[Lemma 5.2 and Corollary 5.4]{DKPS2024} that
        \begin{equation*}
            \Vert\widetilde{S}_{\bfD,U}f\Vert_{\Lp{p}(\reals^{n+m})}\lesssim\Vert f\Vert_{\Lp{p}(U)},\qquad \Vert S_{\bfD,V'}g\Vert_{\Lp{p'}(\reals^{n+m})}\lesssim\Vert g\Vert_{\Lp{p'}(V')},
        \end{equation*}
        yielding the desired bound.
    \end{proof}

    Two-matrix weighted bounds for ``mixed'' paraproducts can be easily deduced from two-matrix weighted bounds for the mixed type operators considered in \cite[Section 8]{DKPS2024}.
    Although at the time of \cite{DKPS2024} only a rather incomplete treatment of the latter bounds was possible, they all follow now readily from \Cref{thm:FeffermanSteinPointwiseWeightedMaximal} and \Cref{thm:FeffermanSteinReducingWeightedMaximal}.
    To show how to apply our results, we give only one example below, the proof for the other operators being similar.
    Note that the parts of the proof that do not rely on the matrix weighted extension of the Fefferman--Stein vector valued inequalities were already carried out in \cite[Section 8]{DKPS2024}.
    Nevertheless, for the reader's convenience we include full details.
    
    As our example, we focus on the mixed paraproduct
    \begin{equation*}
        \Gamma_{\bfD,B}^{(10)}f\coloneq \sum_{\substack{R\in\bfD\\\e\in\cE,~\delta_2\in\cE^2\\\delta_2\neq\e_2}}\frac{1}{\sqrt{|R_2|}}h_{R}^{\e_1,1\oplus\e_2\oplus\delta_2}B_{R}^{\e}\langle f^{\delta_2,2}_{R_2}\rangle_{R_1},
    \end{equation*}
    acting on suitable $\complex^d$-valued functions $f$ on $\reals^{n+m}.$
    The notation and definitions for the other possible paraproducts can be found in the work of Holmes--Petermichl--Wick \cite[Subsection 6.1]{Holmes_Petermichl_Wick_2018},
    together with their duality relations.
    
    \begin{lemma}
        \label{lemma:auxiliary_mixed_operator}
        Let $1<p<\infty$ and let $W$ be a $(d\times d)$ matrix $\bfD$-dyadic biparameter $A_p$ weight on $\reals^n\times\reals^m.$ For (suitable) functions $f:\reals^{n+m}\to\complex,$ let
        \begin{equation*}
            [\widetilde{M}\widetilde{S}]_{\bfD,U}f(x)\coloneq \bigg(\sum_{\substack{R_2\in\cD^2\\\e_2\in\cE^2}}(\sup_{R_1\in\cD^1}|\cW_{R}\langle f^{\e_2,2}_{R_2}\rangle_{R_1}|\1_{R_1}(x_1))^2\frac{\1_{R_2}(x_2)}{|R_2|}\bigg)^{1/2},
        \end{equation*}
        for all $x=(x_1,x_2)\in\reals^{n}\times\reals^{m}.$ Then, we have
        \begin{equation*}
            \Vert[\widetilde{M}\widetilde{S}]_{\bfD,W}f\Vert_{\Lp{p}(\reals^{n+m})}\lesssim_{n,m,p,d}[W]_{A_p,\bfD}^{\beta}\Vert f\Vert_{\Lp{p}(W)},
        \end{equation*}
        where
        \begin{equation*}
            \beta=
            \begin{cases}
                1+\frac{2}{p}+\frac{1}{p-1},\text{ if }p\leq 2\\\\
                \frac{1}{2}+\frac{1}{p}+\frac{2}{p-1},\text{ if }p>2
            \end{cases}
            .
        \end{equation*}
    \end{lemma}

    To prove Lemma~\ref{lemma:auxiliary_mixed_operator}, we need a technical observation already present implicitly in the proof of \cite[Lemma 5.3]{DKPS2024}.

    \begin{lemma}
        \label{lemma:technical_observation}
        Let $1<p<\infty$ and let $\{f_{P,\e}\}_{\substack{P\in\cD^2\\\e\in\cE^2}}$ be a family of nonnegative measurable functions on $\reals^m.$ Then, it holds
        \begin{equation*}
            \bigg\Vert\bigg(\sum_{\substack{P\in\cD^2\\\e\in\cE^2}}\langle f_{P,\e}\rangle_{P}^2\frac{\1_{P}}{|P|}\bigg)^{1/2}\bigg\Vert_{\Lp{p}(\reals^{m})}
            \lesssim_{m,p}\bigg\Vert\bigg(\sum_{\substack{P\in\cD^2\\\e\in\cE^2}} |f_{P,\e}|^2\frac{\1_{P}}{|P|}\bigg)^{1/2}\bigg\Vert_{\Lp{p}(\reals^{m})}.
        \end{equation*}
    \end{lemma}

    \begin{proof}
        We use duality. As usual, by the monotone convergence theorem we may consider only finitely many of the functions $\{f_{P,\e}\}_{\substack{P\in\cD^2\\\e\in\cE^2}}$ to be nonzero. Then, by the dyadic (unweighted) Littlewood--Payley estimates we have
        \begin{equation*}
            \bigg\Vert\bigg(\sum_{\substack{P\in\cD^2\\\e\in\cE^2}}\langle f_{P,\e}\rangle_{P}^2\frac{\1_{P}}{|P|}\bigg)^{1/2}\bigg\Vert_{\Lp{p}(\reals^{m})}\sim_{m,p}\Vert F\Vert_{\Lp{p}(\reals^m)},
        \end{equation*}
        where
        \begin{equation*}
            F\coloneq \sum_{\substack{P\in\cD^2\\\e\in\cE^2}}\langle f_{P,\e}\rangle_{P}h_{P}^{\e}.
        \end{equation*}
        Let $g\in \Lp{p'}(\reals^m)$ be arbitrary. Then, we have
        \begin{align*}
            &\int_{\reals^m}|F(x)g(x)|\mathdm(x)\leq\sum_{\substack{P\in\cD^2\\\e\in\cE^2}}\langle f_{P,\e}\rangle\cdot|g_{P}^{\e}|\\
            &=\sum_{\substack{P\in\cD^2\\\e\in\cE^2}}\int_{\reals^m}f_{P,\e}(x)h_{P}^{\e}(x)g_{P}^{\e}h_{P}^{\e}(x)\mathdm(x)\\
            &\leq\bigg\Vert\bigg(\sum_{\substack{P\in\cD^2\\\e\in\cE^2}}|f_{P,\e}|^2\frac{\1_{P}}{|P|}\bigg)^{1/2}\bigg\Vert_{\Lp{p}(\reals^m)}\cdot
            \bigg\Vert\bigg(\sum_{\substack{P\in\cD^2\\\e\in\cE^2}}|g_{P}^{\e}|^2\frac{\1_{P}}{|P|}\bigg)^{1/2}\bigg\Vert_{\Lp{p'}(\reals^m)}\\
            &\sim_{m,p}\bigg\Vert\bigg(\sum_{\substack{P\in\cD^2\\\e\in\cE^2}}|f_{P,\e}|^2\frac{\1_{P}}{|P|}\bigg)^{1/2}\bigg\Vert_{\Lp{p}(\reals^m)}\cdot\Vert g\Vert_{\Lp{p'}(\reals^m)}.
        \end{align*}
        An appeal to the Riesz representation theorem concludes the proof.
    \end{proof}

    We now prove Lemma~\ref{lemma:auxiliary_mixed_operator}.

    \begin{proof}[Proof of Lemma~\ref{lemma:auxiliary_mixed_operator}]
        First of all, we have
        \begin{align*}
            &\Vert [\widetilde{M}\widetilde{S}]_{\bfD,W}f\Vert_{\Lp{p}(\reals^{n+m})}^{p}=\int_{\reals^n}A(x_1)\mathdm(x_1),
        \end{align*}
        where
        \begin{equation*}
            A(x_1)\coloneq \int_{\reals^m}\bigg(\sum_{\substack{R_2\in\cD^2\\\e_2\in\cE^2}}(\sup_{R_1\in\cD^2}|\cW_{R_1\times R_2}(f_{R_2}^{\e_2,2})_{R_1}|\1_{R_1}(x_1))^2\frac{\1_{R_2}(x_2)}{|R_2|}\bigg)^{p/2}\mathdm(x_2),
        \end{equation*}
        for all $x_1\in\reals^{m}.$
        
        Fix $x_1\in\reals^{m}.$ For a.e.~$x_2\in\reals^n,$ we denote by $\cW_{x_2,R_1}$ the reducing operator of the weight $W_{x_2}(y)\coloneq W(y,x_2),$ $y\in\reals^n$ over any $R_1\in\cD^1$ with respect to the exponent $p.$ For fixed $R_1\in\cD^1,$ we define $W_{R_1}(x_2)\coloneq \cW_{x_2,R_1}^{p},$ for a.e.~$x_2\in\reals,$ and denote by $\cW_{R_1,R_2}$ the reducing operator of $W_{R_1}$ over any $R_2\in\cD^2$ with respect to the exponent $p.$ Let us observe here in general that if $\cV_{E}$ is the reducing operator of a matrix weight $V$ over a set $E$ with respect to exponent $p,$ and $V^{-1/(p-1)}$ is also integrable over $E$ with reducing exponent $\cV'_{E}$ with respect to exponent $p'$, then
        \begin{equation}
            \label{replace inverse by prime_1}
            |\cV_{E} e|\lesssim_{p,d}|\cV_{E}\cV_{E}'|\cdot{\langle|V^{1/p}e|\rangle}_{E},\quad\forall e\in\complex^{d};
        \end{equation}
        for the proof of this fact see, for instance, \cite[Lemmas~3.1 and~3.2]{DKPS2024}. Applying now first \eqref{iterate reducing operators}, then \eqref{replace inverse by prime_1}, and finally \eqref{uniform domination of characteristics of averages}, we obtain
        \begin{align*}
            &\sup_{R_1\in\cD^1}|\cW_{R_1\times R_2}\langle f_{R_2}^{\e_2,2}\rangle_{R_1}|\1_{R_1}(x_1)\lesssim_{p,d}\sup_{R_1\in\cD^1}|\cW_{R_1,R_2}\langle f_{R_2}^{\e_2,2}\rangle_{R_1}|\1_{R_1}(x_1)\\
            &\lesssim_{p,d}
            \sup_{R_1\in\cD^1}[W_{R_1}]_{A_p,\cD^2}^{\frac{1}{p}}\left\langle|W_{R_1}^{1/p}\langle f_{R_2}^{\e_2,2}\rangle_{R_1}|\right\rangle_{R_2}\1_{R_1}(x_1)\\
            &\lesssim_{p,d}[W]_{A_p,\bfD}^{\frac{1}{p}}\bigg\langle\sup_{R_1\in\cD^1}|W_{R_1}^{1/p}\langle f_{R_2}^{\e_2,2}\rangle_{R_1}|\1_{R_1}(x_1)\bigg\rangle_{R_2}
        \end{align*}
        for all $\e_2\in\cE^2$ and $R_2\in\cD^2.$
        So we get
        \begin{align*}
            A(x_1)\lesssim_{p,d}[W]_{A_p,\bfD}B(x_1),
        \end{align*}
        where
        \begin{align*}
            B(x_1)\coloneq 
            \int_{\reals^m}\bigg(\sum_{\substack{R_2\in\cD^2\\\e_2\in\cE^2}}\bigg\langle\sup_{R_1\in\cD^1}|W_{R_1}^{1/p}\langle f_{R_2}^{\e_2,2}\rangle_{R_1}|\1_{R_1}(x_1)\bigg\rangle_{R_2}^2\frac{\1_{R_2}(x_2)}{|R_2|}\bigg)^{p/2}\mathdm(x_2).
        \end{align*}
        Since $x_1$ is fixed, by Lemma~\ref{lemma:technical_observation} we obtain
        \begin{align*}
            B(x_1)&\leq 
            \int_{\reals^m}\bigg(\sum_{\substack{R_2\in\cD^2\\\e_2\in\cE^2}}\bigg(\sup_{R_1\in\cD^1}|W_{R_1}^{1/p}\langle f_{R_2}^{\e_2,2}\rangle_{R_1}|\1_{R_1}(x_1)\bigg)^2\frac{\1_{R_2}(x_2)}{|R_2|}\bigg)^{p/2}\mathdm(x_2)\\
            &=\int_{\reals^m}\bigg(\sum_{\substack{R_2\in\cD^2\\\e_2\in\cE^2}}\bigg(\widetilde{M}_{W_{x_2},\cD^1}(f_{R_2}^{\e_2,2})(x_1)\bigg)^2\frac{\1_{R_2}(x_2)}{|R_2|}\bigg)^{p/2}\mathdm(x_2).
        \end{align*}
        Thus, by Fubini--Tonelli we have
        \begin{align*}
            &\int_{\reals^n}B(x_1)\mathdm(x_1)\\
            &\leq\int_{\reals^m}\bigg(\int_{\reals^n}\bigg(\sum_{\substack{R_2\in\cD^2\\\e_2\in\cE^2}}\bigg(\widetilde{M}_{W_{x_2},\cD^1}(f_{R_2}^{\e_2,2})(x_1)\bigg)^2\frac{\1_{R_2}(x_2)}{|R_2|}\bigg)^{p/2}\mathdm(x_1)\bigg)\mathdm(x_2).
        \end{align*}
        For a.e.~$x_2\in\reals^m,$ using \Cref{thm:FeffermanSteinReducingWeightedMaximal} in the first step and \eqref{uniform slicing variable} in the second step, we obtain
        \begin{align*}
            &\int_{\reals^n}\bigg(\sum_{\substack{R_2\in\cD^2\\\e_2\in\cE^2}}\bigg(\widetilde{M}_{W_{x_2},\cD^1}(f_{R_2}^{\e_2,2})(x_1)\bigg)^2\frac{\1_{R_2}(x_2)}{|R_2|}\bigg)^{p/2}\mathdm(x_1)\\
            &\lesssim_{m,p,d}[W_{x_2}]_{A_p,\cD^1}^{1+\max\{p,p'\}}\int_{\reals^m}\bigg(\sum_{\substack{R_2\in\cD^2\\\e_2\in\cE^2}}|W_{x_2}^{1/p}(x_1)f_{R_2}^{\e_2,2}(x_1)|^2\frac{\1_{R_2}(x_2)}{|R_2|}\bigg)^{p/2}\mathdm(x_1)\\
            &\lesssim_{p,d}[W]_{A_p,\bfD}^{1+\max\{p,p'\}}\int_{\reals^m}\bigg(\sum_{\substack{R_2\in\cD^2\\\e_2\in\cE^2}}|W(x_1,x_2)^{1/p}f_{R_2}^{\e_2,2}(x_1)|^2\frac{\1_{R_2}(x_2)}{|R_2|}\bigg)^{p/2}\mathdm(x_1)\\
            &=[W]_{A_p,\bfD}^{1+\max\{p,p'\}}\int_{\reals^m}\bigg(S_{\cD^2,W_{x_1}}(f(x_1,\var))(x_2)\bigg)^{p/2}\mathdm(x_1),
        \end{align*}
        where we denote $W_{x_1}\coloneq W(x_1,x_2),$ $x_1\in\reals^n.$ Fubini--Tonelli now yields
        \begin{align*}
            &\int_{\reals^n}B(x_1)\mathdm(x_1)\\
            &\lesssim_{m,p,d}[W]_{A_p,\bfD}^{1+\max\{p,p'\}}\int_{\reals^n}\bigg(\int_{\reals^m}\bigg(S_{\cD^2,W_{x_1}}(f(x_1,\var))(x_2)\bigg)^{p/2}\mathdm(x_2)\bigg)\mathdm(x_1).
        \end{align*}
        For a.e.~$x_1\in\reals^n,$ applying first the matrix weighted bounds for the one-parameter dyadic square function from \cite{Isralowitz_2020} and then \eqref{uniform slicing variable}, we get
        \begin{align*}
            &\int_{\reals^m}\bigg(S_{\cD^2,W_{x_1}}(f(x_1,\var))(x_2)\bigg)^{p/2}\mathdm(x_2)\\
            &\lesssim_{n,p,d}[W]_{A_p,\bfD}^{\max\left\{p',\frac{p}{2}+\frac{1}{p-1}\right\}}\int_{\reals^m}|W(x_1,x_2)^{1/p}f(x_1,x_2)|^{p}\mathdm(x_2).
        \end{align*}
        Putting the above estimates together, we finally deduce
        \begin{align*}
             &\Vert [\widetilde{M}\widetilde{S}]_{\bfD,W}f\Vert_{\Lp{p}(\reals^{n+m})}^{p}
             \lesssim_{n,m,p,d}[W]_{A_p,\bfD}^{\alpha}\int_{\reals^{n+m}}|W(x)^{1/p}f(x)|^{p}\mathdm(x),
        \end{align*}
        where
        \begin{equation*}
            \alpha=1+1+\max\{p,p'\}+\max\left\{p',\frac{p}{2}+\frac{1}{p-1}\right\},
        \end{equation*}
        concluding the proof.
    \end{proof}

    We now prove two-matrix weighted bounds for the mixed paraproducts.

    \begin{proposition}
        \label{prop:mixedparaproducts}
        Let $d,p,U,V$ and $B$ be as above. Then there holds
        \begin{equation*}
            \Vert\Gamma_{\bfD,B}^{(10)}\Vert_{\Lp{p}(U)\rightarrow \Lp{p}(V)}\lesssim\Vert B\Vert_{\emph{BMO}_{\emph{prod},\bfD}(U,V,p)},
        \end{equation*}
        where all implied constants depend only on $n,m,d,p,[U]_{A_p,\bfD}$ and $[V]_{A_p,\bfD}.$
    \end{proposition}

    \begin{proof}
        Throughout the proof $\lesssim,\gtrsim,\sim$ mean that all implied inequality constants depend only on $n,m,d,[U]_{A_p,\bfD}$ and $[V]_{A_p,\bfD}.$

        We adapt the factorization trick from the proof of \cite[Proposition 6.1]{Holmes_Petermichl_Wick_2018}. We have
        \begin{align*}
            \left|(\Gamma_{\bfD,B}^{(10)}f,g)\right|&=\bigg|\sum_{\substack{R\in\bfD\\\e\in\cE,~\delta_2\in\cE^2\\\delta_2\neq\e_2}}\frac{1}{\sqrt{|R_2|}}\langle B_{R}^{\e}\langle f^{\delta_2,2}_{R_2}\rangle_{R_1},g_{R}^{\e_1,1\oplus\e_2\oplus\delta_2}\rangle\bigg|\\
            &=\bigg|\sum_{\substack{R\in\bfD\\\e\in\cE}}\mathrm{tr}((B_{R}^{\e})^{\ast}\Phi_{R}^{\e})\bigg|\lesssim\Vert B\Vert_{\BMOprodD(U,V,p)}\Vert\Phi\Vert_{\mathrm{H}^1_{\bfD}(U,V,p)},
        \end{align*}
        where
        \begin{equation*}
            \Phi_{R}^{\e}\coloneq \frac{1}{\sqrt{|R_2|}}\sum_{\delta_2\in\cE^2\setminus\{\e_2\}}g^{\e_1,1\oplus\e_2\oplus\delta_2}_{R}\overline{\langle f^{\delta_2,2}_{R_2}\rangle_{R_1}}^{T},\quad R\in\bfD,~\e\in\cE.
        \end{equation*}
        It suffices now to prove that
        \begin{equation}
            \label{goal1}
            \Vert\Phi\Vert_{\mathrm{H}^1_{\bfD}(U,V,p)}\lesssim\Vert f\Vert_{\Lp{p}(U)}\Vert g\Vert_{\Lp{p'}(V')}.
        \end{equation}
        We can bound $\Vert\Phi\Vert_{\mathrm{H}^1_{\bfD}(U,V,p)}$ by
        \begin{align*}
            & 2^{m}\sum_{\substack{R\in\bfD\\\e\in\cE,~\delta_2\in\cE^2\\\delta_2\neq\e_2}}|V(x)^{-1/p}g^{\e_1,1\oplus\e_2\oplus\delta_2}_{R}\overline{\langle f^{\delta_2,2}_{R_2}\rangle_{R_1}}^{T}\cU_{R}|^2\frac{\1_{R}(x)}{|R_1|\cdot|R_2|^2}\\
            &\leq 2^{m}\sum_{\substack{R\in\bfD\\\e\in\cE,~\delta_2\in\cE^2\\\delta_2\neq\e_2}}|V(x)^{-1/p}g^{\e}_{R}|^2\frac{\1_{R}(x)}{|R|}\cdot|\cU_{R}\langle f^{\delta_2,2}_{R_2}\rangle_{R_1}|^2\frac{\1_{R_2}(x_2)}{|R_2|}\\
            &\leq 2^m S_{\bfD,V'}g(x)[\widetilde{M}\widetilde{S}]_{\bfD,U}f(x),
        \end{align*}
        where $S_{\bfD,V'}g$ and $[\widetilde{M}\widetilde{S}]_{\bfD,U}f$ are defined as previously.
        Therefore
        \begin{equation*}
            \Vert\Phi\Vert_{\mathrm{H}^1_{\bfD}(U,V,p)}\leq\Vert S_{\bfD,V'}g\Vert_{\Lp{p'}(\reals^{n+m})}\Vert[\widetilde{M}\widetilde{S}]_{\bfD,U}f\Vert_{\Lp{p}(\reals^{n+m})}.
        \end{equation*}
        It is proved in \cite[Lemmas 5.2]{DKPS2024} that
        \begin{equation*}
           \Vert S_{\bfD,V'}g\Vert_{\Lp{p'}(\reals^{n+m})}\lesssim\Vert g\Vert_{\Lp{p'}(V')}.
        \end{equation*}
        Moreover, Lemma~\ref{lemma:auxiliary_mixed_operator} yields
        \begin{align*}
            \Vert[\widetilde{M}\widetilde{S}]_{\bfD,U}f\Vert_{\Lp{p}(\reals^{n+m})}
            \lesssim\Vert f\Vert_{\Lp{p}(U)},
        \end{align*}
        proving \eqref{goal1}.
    \end{proof}

    \subsection{Two matrix weighted upper bounds for bicommutators}
    \label{subsec:bicommutators}

    Two matrix weighted upper bounds for biparameter paraproducts lead naturally to two matrix weighted upper bounds for bicommutators. Here we overview very briefly the simplest case: that of bicommutators with Haar multipliers on $\reals.$ The argument is a straightforward adaptation of the one in the (unweighted) $\Lp{2}$ scalar valued case from \cite{Blasco_Pott_2005}.
    
    Let $\bfD=\cD^1\times\cD^2$ be a biparameter dyadic grid on $\reals^2,$ and let $\sigma_1=\{\sigma_1(I)\}_{I\in\cD^1}$ and $\sigma_2=\{\sigma_2(J)\}_{J\in\cD^2}$ be (finitely supported) sequences in $\{-1,0,1\}.$ For any function $f\in \Lp{2}(\reals;\complex^d)$ we define
    \begin{equation*}
        T_{\sigma_1}f \coloneq  \sum_{I\in\cD^1}\sigma_1(I)h_{I}f_{I},\quad T_{\sigma_2}f \coloneq  \sum_{J\in\cD^2}\sigma_2(J)h_{J}f_{J}.
    \end{equation*}
    We can then consider the operators $T^1_{\sigma_1}$ and $T^2_{\sigma_2}$ acting on functions $f\in \Lp{2}(\reals^2;\complex^d)$ by
    \begin{equation*}
        T^1_{\sigma_1}f(x_1,x_2) \coloneq  T_{\sigma_1}(f(\var,x_2))(x_1),\quad T^2_{\sigma_2}f(x_1,x_2) \coloneq  T_{\sigma_2}(f(x_1,\var))(x_2),
    \end{equation*}
    for a.e.~$(x_1,x_2)\in\reals^2.$

    Let $B:\reals^{2}\to \mathrm{M}_{d}(\complex)$ be a locally integrable function. For scalar valued locally integrable functions $b$ on $\reals^2$ it is shown in \cite{Blasco_Pott_2005} that
    \begin{equation}
        \label{Blasco_Pott_identity}
        [T^1_{\sigma_1},[T^2_{\sigma_2},b]] = [T^1_{\sigma_1},[T^2_{\sigma_2},\Lambda_{b}]],
    \end{equation}
    where $\Lambda_{b}$ is the so-called \emph{symmetrized paraproduct} given by
    \begin{equation*}
        \Lambda_{b}f \coloneq  \Pi^{(11)}_{b}f+\Pi^{(10)}_{b}f+\Pi^{(01)}_{b}f+\Pi^{(00)}_{b}f.
    \end{equation*}
    Applying \eqref{Blasco_Pott_identity} entrywise we deduce
    \begin{equation*}
        [T^1_{\sigma_1},[T^2_{\sigma_2},B]] = [T^1_{\sigma_1},[T^2_{\sigma_2},\Lambda_{B}]],
    \end{equation*}
    where the symmetrized paraproduct $\Lambda_{B}$ is given by
    \begin{equation*}
        \Lambda_{B}f \coloneq  \Pi^{(11)}_{B}f+\Pi^{(10)}_{B}f+\Pi^{(01)}_{B}f+\Pi^{(00)}_{B}f.
    \end{equation*}
    Let now $1<p<\infty$ and let $U,V$ be biparameter $(d\times d)$ matrix $\bfD$-dyadic $A_p$ weights on $\reals\times\reals.$ In the following estimates, all implied constants depend only on $d,p,[U]_{A_p,\bfD}$ and $[V]_{A_p,\bfD}.$ Using \eqref{uniform slicing variable} and the well-known two matrix weighted bounds in the one parameter setting \cite{Cruz_Uribe_OFS_2018} we deduce
    \begin{equation*}
        \Vert T^{j}_{\sigma_{j}}\Vert_{\Lp{p}(U)\to \Lp{p}(V)}\lesssim1,\quad j=1,2.
    \end{equation*}
    Observe also that Proposition~\ref{prop:pureparaproducts} and Proposition~\ref{prop:mixedparaproducts} immediately yield
    \begin{equation*}
    \Vert\Lambda_{B}\Vert_{\Lp{p}(U)\to \Lp{p}(V)}\lesssim1.
    \end{equation*}
    Thus, we obtain
    \begin{align*}
        &\Vert[T_1,[T_2,B]]\Vert_{\Lp{p}(U)\to \Lp{p}(V)} = \Vert[T_1,[T_2,\Lambda_{B}]]\Vert_{\Lp{p}(U)\to \Lp{p}(V)}\\
        &\leq\Vert T_1T_2\Lambda_{B}\Vert_{\Lp{p}(U)\to \Lp{p}(V)}+\Vert T_1\Lambda_{B}T_2\Vert_{\Lp{p}(U)\to \Lp{p}(V)}\\
        &+\Vert T_2\Lambda_{B}T_1\Vert_{\Lp{p}(U)\to \Lp{p}(V)}+\Vert\Lambda_{B}T_1T_2\Vert_{\Lp{p}(U)\to \Lp{p}(V)}\lesssim1.
    \end{align*}

    \appendix
    \section{Appendix}
    \label{s:appendix}

    Through the article we have used the results of Bownik and Cruz-Uribe~\cite{Cruz_Uribe_Bownik_Extrapolation} in the complex setting.
    The results in that paper are stated in the real setting and
    it is not immediate that they work in the complex setting.
    Some modifications are necessary and some steps require proper justification.
    For instance, given a family of norms measurably parametrized (over $\complex^d$),
    it is not obvious that one can assign them a reducing matrix and a complex John ellipsoid
    in a measurable way.
    This is shown in detail in~\cite[Appendix~A]{DKPS2024}.
    We devote this appendix to briefly discussing the necessary modifications
    and where to find the details to get the results of~\cite{Cruz_Uribe_Bownik_Extrapolation}
    for complex valued matrix weights.

    \subsection{Convex sets and seminorms in the complex setting}
    \label{subsec:ComplexConvexAndSeminorms}
    We begin with the required changes in~\cite[Section~2]{Cruz_Uribe_Bownik_Extrapolation}.
    In both the present work and that of Bownik and Cruz-Uribe,
    one considers symmetric sets.
    The difference is that we substitute real symmetric sets $E \subseteq \reals^d,$
    that is sets $E$ such that $-u \in E$ for every $u \in E,$
    by complex symmetric sets $E \subseteq \complex^d,$
    i.e. sets $E$ for which $\lambda u \in E$ for every $u \in E$ and for any $\lambda\in\complex$ with $|\lambda| = 1.$
    Of course, the definitions for seminorms are also the same,
    only that we consider them to be defined on the vector space $\complex^d$
    and to be complex homogeneous functions, which means that if $\rho$
    is a seminorm on $\complex^d,$ then $\rho(\lambda v) = |\lambda| \rho(v)$
    for any $v \in \complex^d$ and $\lambda \in \complex.$

    Dual seminorms, their properties and their relation to
    the polar of convex (complex) symmetric sets follow the same exposition
    in the vector space $\complex^d$ as in~\cite{Cruz_Uribe_Bownik_Extrapolation}.
    Here one only needs to keep in mind that we substitute the real Euclidean product
    by the complex Hermitian product of $\complex^d.$
    For this reason, in this context the support function $h_K$ of
    a set $K \in \convex(\complex^d)$ has to be defined as
    \begin{equation*}
        h_K(v) \coloneq \sup_{w \in K} |\langle v,w \rangle|,
    \end{equation*}
    (cf.~\cite[Definition~2.12]{Cruz_Uribe_Bownik_Extrapolation}).

    For the rest of the section, the facts about positive definite matrices follow
    by the same arguments,
    since the facts that the authors of~\cite{Cruz_Uribe_Bownik_Extrapolation} use
    hold in the complex setting (see~\cite[Chapter~6]{Bhatia_Positive_definite_matrices_2007}).

    \subsection{Convex-set valued functions in the complex setting}

    Here we briefly explain the main relevant adaptations that have to be performed in Section 3 of \cite{Cruz_Uribe_Bownik_Extrapolation} in the complex setting.
    Most of the results and arguments in that work can be ported to the complex setting in a straightforward way.
    The general principle that this relies on is the fact that one can identify $\complex^d$ with $\reals^{2d},$ being both equivalent as topological spaces, as metric spaces, as measure spaces and as real vector spaces.
    This, of course, requires implicitly other minor modifications, such as substituting the real Euclidean scalar product by the Hermitian scalar product.

    The application of this general principle together with the work in~\cite{Cruz_Uribe_Bownik_Extrapolation} already yields the necessary facts about the measurability of $\mathcal{K}(\complex^d)$ valued maps and selection functions.
    Statements about measurability of closed convex hulls of countable unions and of countable intersections of convex body valued functions follow similarly.
    Moreover, a similar remark holds for the measurability of the polar map.

    The point that is not immediate to adapt, and which is one of the keystones in the development of the results in~\cite{Cruz_Uribe_Bownik_Extrapolation} is the existence and properties of John ellipsoids for bounded complex symmetric convex subsets of $\complex^d$. These are thoroughly established in~\cite[Appendix~A]{DKPS2024} and the distinction between complex and real ellipsoids is made precise.  In fact, it is implicit in \cite[Subsection~A.4.7]{DKPS2024} that a real ellipsoid in $\reals^{2d}$ is a complex ellipsoid in $\complex^d$ if and only if it is invariant under a certain family of real orthogonal maps $\widetilde{L}_{z}.$

    Denote by $\mathbf{\overline{B}}_{\reals^{2d}}$ the closed unit ball in $\reals^{2d}$ and by $\mathbf{\overline{B}}_{\complex^{d}}$ the closed unit ball in $\complex^{d}.$
    Recall that a \emph{real ellipsoid} in $\reals^{2d}$ is a subset $E$ of $\reals^{2d}$ of the form $E=A\mathbf{\overline{B}}_{\reals^{2d}}$ for some real linear map $A:\reals^{2d}\to\reals^{2d}$.
    It can be seen that, for a real ellipsoid $E,$ the linear map $A$ with $E=A\mathbf{\overline{B}}_{\reals^{2d}}$ can be taken to be positive semidefinite and the choice of $A$ is unique under this condition.  Similarly, a \emph{complex ellipsoid} in $\complex^{d}$ is a subset $E$ of $\complex^{d}$ of the form $E=A\mathbf{\overline{B}}_{\complex^d}$ for some complex linear map $A:\complex^{d}\to\complex^{d}$. It can also be seen that for each complex ellipsoid $E$ there is a unique complex positive semidefinite linear map $A$ with $E=A\mathbf{\overline{B}}_{\complex^d}.$ A sketch of the real case can be found in \cite[page 304]{Goffin1983}. For completeness we include a detailed proof in the complex case, adapting the ideas in \cite[page 304]{Goffin1983}.

    \begin{lemma}
        \label{lem:matrix_ellipsoid}
        Let $A,C\in M_{d}(\complex)$. Consider the complex ellipsoid $E:=A\mathbf{\overline{B}}_{\complex^{d}}$. Then, we have $E=C\mathbf{\overline{B}}_{\complex^{d}}$ if and only if $AA^{\ast}=CC^{\ast}$. In particular, $H := |A^{\ast}| = (AA^{\ast})^{1/2}$ is the unique positive semidefinite matrix such that $E=H\mathbf{\overline{B}}_{\complex^d}$.
    \end{lemma}

    \begin{proof}
        Since $AA^{\ast}$ and $CC^{\ast}$ are both self-adjoint, we have $AA^{\ast}=CC^{\ast}$ if and only if
        \begin{equation*}
            \langle AA^{\ast}x,x\rangle = \langle CC^{\ast}x,x\rangle,\quad\forall x\in\complex^{d},
        \end{equation*}
        that is
        \begin{equation*}
            |A^{\ast}x|=|C^{\ast}x|,\quad\forall x\in\complex^{d}.
        \end{equation*}
        
        Observe that for all $x\in\complex^{d}$ it holds
        \begin{align*}
            |A^{\ast}x| &= \sup\{|\langle A^{\ast}x,y\rangle|:~y\in\mathbf{\overline{B}}_{\complex^{d}}\}
            = \sup\{|\langle x,Ay\rangle|:~y\in\mathbf{\overline{B}}_{\complex^{d}}\}\\
            &= \sup\{|\langle x,z\rangle|:~z\in A\mathbf{\overline{B}}_{\complex^{d}}\}
        \end{align*}
        and similarly
        \begin{equation*}
            |C^{\ast}x| 
            = \sup\{|\langle x,z\rangle|:~z\in C\mathbf{\overline{B}}_{\complex^{d}}\}.
        \end{equation*}
        Therefore, if $A\mathbf{\overline{B}}_{\complex^{d}} = C\mathbf{\overline{B}}_{\complex^{d}}$, then it is clear that $|A^{\ast}x|=|C^{\ast}x|$, for all $x\in\complex^{d}$ and thus $AA^{\ast}=CC^{\ast}$.
        
        Conversely, assume that $AA^{\ast}=CC^{\ast}$. Then $|A^{\ast}|=(AA^{\ast})^{1/2}=(CC^{\ast})^{1/2}=|C^{\ast}|$. By the polar decomposition we can write $A^{\ast}=U|A^{\ast}|$ and $C^{\ast}=V|C^{\ast}|$ for some unitary matrices $U,V$, therefore
        \begin{align*}
            A\mathbf{\overline{B}}_{\complex^{d}}=
            |A^{\ast}|U^{\ast}\mathbf{\overline{B}}_{\complex^{d}}
            =|A^{\ast}|\mathbf{\overline{B}}_{\complex^{d}}
            =|C^{\ast}|\mathbf{\overline{B}}_{\complex^{d}}
            =|C^{\ast}|V^{\ast}\mathbf{\overline{B}}_{\complex^{d}}
            =C\mathbf{\overline{B}}_{\complex^{d}}.
        \end{align*}

        Finally, the last claim follows immediately from the fact that $HH^{\ast}=AA^{\ast}$.
    \end{proof}
    
    Let us denote by $R:\complex^{d}\to\reals^{2d}$ the map given by
    \begin{equation*}
    	R(x_1+ix_2,\ldots,x_{2d-1}+ix_{2d}):=(x_1,x_2,\ldots,x_{2d-1},x_{2d}).
    \end{equation*}
    If $A:\complex^{d}\to\complex^{d}$ is any complex linear map, then $A$ corresponds to the real linear map $\widetilde{A} = RAR^{-1}:\reals^{2d}\to\reals^{2d}.$
    Now consider for each $z\in\complex$ the map $L_{z}:\complex^{d}\to\complex^{d}$ denoting componentwise multiplication with $z.$
    Then we can consider the corresponding linear map $\widetilde{L}_{z} \coloneq RL_{z}R^{-1}:\reals^{2d}\to\reals^{2d}.$
    Observe that the matrix corresponding to $\widetilde{L}_{z}$ is a $(2d \times 2d)$-matrix with $d$ blocks of the form
    \begin{equation*}
        \begin{bmatrix}
        	\mathrm{Re}(z) & -\mathrm{Im}(z)\\
        	\mathrm{Im}(z) & \mathrm{Re}(z)
        \end{bmatrix}
    \end{equation*}
    on the diagonal.
    With the notation that we have introduced, we have that for each $z\in\complex$ and for each nonempty $E \subseteq \complex^d$ it holds that $zE=E$ if and only if $\widetilde{L}_{z}(R(E))=R(E).$
    Moreover, if $z\in\complex$ with $|z|=1,$ then $\widetilde{L}_{z}$ is a real orthogonal map.
    The next lemma states the precise relation between real and complex ellipsoids.

     \begin{lemma}
         \label{lem:complex_vs_real_ellipsoids}
          Let $\mathcal{C}_{2d}(\reals)$ be the set of all real ellipsoids in $\reals^{2d}$ and let $\mathcal{C}_{d}(\complex)$ be the set of all complex ellipsoids in $\complex^{d}$. Then we have
     \begin{equation}
     	\label{eq:complex_ellipsoids}
     	\mathcal{C}_{d}(\complex) = \{R^{-1}(E):~E\in\mathcal{C}_{2d}(\reals^{2d})\text{ and }\widetilde{L}_{z}(E)=E \text{ for each } |z| = 1\}.
     \end{equation}
     In particular, the complex ellipsoids in $\complex^d$ form a closed subset of the real ellipsoids in $\reals^{2d}$.
     \end{lemma}
     
    \begin{proof}   
     Here, for a real ellipsoid $E$ in $\reals^{2d}$ we denote by $M_\reals(E)$ the unique positive semidefinite matrix with $E = M_{\reals}(E)\mathbf{\overline{B}}_{\reals^{2d}}.$
     Likewise, for a complex ellipsoid $E$ in $\complex^d,$ we use $M_\complex$ to denote the unique complex positive semidefinite matrix with $E = M_{\complex}(E)\mathbf{\overline{B}}_{\complex^{d}}.$
     
     First of all, let $E$ be a real ellipsoid in $\reals^{2d}$  with $\widetilde{L}_{z}(E)=E$. Then, for each $z\in\complex$ with $|z|=1$ we have
     \begin{align*}
     	E&=\widetilde{L}_{z}E=\widetilde{L}_{z}M_{\reals}(E)\mathbf{\overline{B}}_{\reals^{2d}}
        =(\widetilde{L}_{z}M_{\reals}(E)(\widetilde{L}_{z}M_{\reals}(E))^{*})^{1/2}\mathbf{\overline{B}}_{\reals^{2d}}\\
     	&=\widetilde{L}_{z}M_{\reals}(E)\widetilde{L}_{z}^{-1}\mathbf{\overline{B}}_{\reals^{2d}},
     \end{align*}
     because $\widetilde{L}_{z}$ is orthogonal, where in the third equality we have used the polar decomposition of $(\widetilde{L}_{z}M_{\reals}(E))^{*}$ as in the proof of Lemma~\ref{lem:matrix_ellipsoid}. Since $\widetilde{L}_{z}M_{\reals}(E)\widetilde{L}_{z}^{-1}$ is positive semidefinite, uniqueness of $M_{\reals}(E)$ implies that
     \begin{equation*}
     	M_{\reals}(E) = \widetilde{L}_{z}M_{\reals}(E)\widetilde{L}_{z}^{-1} = \widetilde{L}_{z}^{-1}M_{\reals}(E)\widetilde{L}_{z}.
     \end{equation*}
     Now we need to check that $R^{-1}(E)$ is a complex ellipsoid. We compute
     \begin{equation*}
     	R^{-1}E = R^{-1}M_{\reals}(E)\mathbf{\overline{B}}_{\reals^{2d}} = R^{-1}M_{\reals}(E)R\mathbf{\overline{B}}_{\complex^{d}}.
     \end{equation*}
     Thus, in order to show that $R^{-1}E\in\mathcal{C}_{d}(\complex),$ it suffices to show that the real linear map $R^{-1}M_{\reals}(E)R:\complex^{d}\to\complex^{d}$ is in fact complex linear. To this end, it suffices to prove that
     \begin{equation*}
     	L_{z}R^{-1}M_{\reals}(E)R = R^{-1}M_{\reals}(E)RL_{z}
     \end{equation*}
     for all $z\in\complex$ with $|z|=1$. For such $z$ we compute
     \begin{align*}
     	&L_{z}R^{-1}M_{\reals}(E)RL_{z}^{-1}=L_{z}R^{-1}\widetilde{L}_{z}^{-1}M_{\reals}(E)\widetilde{L}_{z}RL_{z}^{-1}\\
     	&=L_{z}R^{-1}RL_{z}^{-1}R^{-1}M_{\reals}(E)RL_{z}R^{-1}RL_{z}^{-1}\\
     	&=R^{-1}M_{\reals}(E)R.
     \end{align*}
	This proves the inclusion $\supseteq$ in \eqref{eq:complex_ellipsoids}.
	
	We now show the inclusion $\subseteq$ in \eqref{eq:complex_ellipsoids}. Consider a complex ellipsoid $E$ in $\complex^{d},$ so that $R(E)$ is a real ellipsoid in $\reals^{2d}$ because
	\begin{align*}
		R(E) = RM_{\complex}(E)R^{-1}\mathbf{\overline{B}}_{\reals^{2d}}.
	\end{align*}
    Moreover, for all $z\in\complex$ with $|z|=1$ we have
     \begin{align*}
     	\widetilde{L}_{z}R(E) &= \widetilde{L}_{z}RM_{\complex}(E)R^{-1}\mathbf{\overline{B}}_{\reals^{2d}} = RL_{z}R^{-1}RM_{\complex}(E)R^{-1}\mathbf{\overline{B}}_{\reals^{2d}}\\
     	&= RL_{z}M_{\complex}(E)R^{-1}\mathbf{\overline{B}}_{\reals^{2d}} = RM_{\complex}(E)L_{z}R^{-1}\mathbf{\overline{B}}_{\reals^{2d}}
     	= RM_{\complex}(E)L_{z}\mathbf{\overline{B}}_{\complex^{d}}\\
     	&=
     	RM_{\complex}(E)\mathbf{\overline{B}}_{\complex^d} = R(E),
     \end{align*}
    where we used the fact that $L_{z}M_{\complex}(E) = M_{\complex}(E)L_{z}$ by linearity and the fact that the closed unit ball of $\complex^{d}$ is invariant under multiplication with $z.$
    This also shows that the complex ellipsoids in $\complex^d$ form a closed subset of the real ellipsoids in $\reals^{2d},$
    since orthogonal maps induce isometries with respect to the Hausdorff metric.
    \end{proof}
    
     The previous exposition is related to~\cite[Theorem~3.7]{Cruz_Uribe_Bownik_Extrapolation}.
    The idea behind this result is the fact that
    for any measurable convex body valued function $F,$
    the function $G(x)$ defined as the John ellipsoid of the convex body $F(x)$
    is also measurable.
    For the reader's convenience, we include the precise statement
    in the context of complex convex body valued functions.
    
    \begin{theorem}[Theorem~3.7 in \cite{Cruz_Uribe_Bownik_Extrapolation}]
        Suppose that $F\colon \Omega \rightarrow \convex(\complex^d)$ is measurable.
        Then there exists a measurable matrix-valued mapping $W\colon \Omega \rightarrow \matrices{d}{\complex}$ such that
        \begin{enumerate}[(i)]
            \item
            the columns of $W(x)$ are mutually orthogonal,
            \item 
            for every $x \in \Omega,$ it holds that
            \begin{equation*}
                W(x) \mathbf{\overline{B}}_{\complex^d} \subseteq F(x)
                \subseteq \sqrt{d} W(x) \mathbf{\overline{B}}_{\complex^d}.
            \end{equation*}
        \end{enumerate}
    \end{theorem}

    The proof of this result is based on \cite[Lemma~3.8]{Cruz_Uribe_Bownik_Extrapolation},
    \cite[Lemma~3.9]{Cruz_Uribe_Bownik_Extrapolation}
    and \cite[Lemma~3.10]{Cruz_Uribe_Bownik_Extrapolation},
    in a way analogous to that of the real convex body context.
    We include the statements of these lemmata without proofs, since they are the same as in the real context once one has Lemma~\ref{lem:complex_vs_real_ellipsoids}.
    However, we mention some remarks about each of them.

    \begin{lemma}[Lemma~3.8 in \cite{Cruz_Uribe_Bownik_Extrapolation}]
        \label{lemma:JohnEllipsoidMeasurability}
        Given a measurable convex body valued function $F\colon \Omega \rightarrow \convex(\complex^d)$ such that $F(x)$ has nonempty interior for every $x \in \Omega,$
        there exists a measurable convex body valued function $G\colon \Omega \rightarrow \convex(\complex^d)$ such that $G(x)$ is an ellipsoid with nonempty interior and with
        \begin{equation}
            \label{eq:JohnEllipsoidInclusions}
            G(x) \subseteq F(x) \subseteq \sqrt{d} G(x)
        \end{equation}
        for every $x \in \Omega.$
    \end{lemma}

    \begin{remark}
        As mentioned before, the proof in the complex case follows the same arguments as in the real case almost verbatim, only with the obvious modifications.
        It is worth mentioning that the proof makes use of the Blaschke Selection Theorem~\cite[Theorem~1.8.7]{Schneider_Convex_Bodie_1993},
        which is equally valid in the complex setting.
        
        Also, this result can be shown by another approach.
        One can note that the map sending each convex body with nonempty interior to its John ellipsoid is continuous with respect to the Hausdorff distance. In the real case, \cite[Subsection~A.4.3]{DKPS2024} appeals directly to~\cite{Mordhorst_2017} for this result. An extension to the complex setting is performed in detail in \cite[Subsection~A.4.7]{DKPS2024}. Combining this with the complex version of \cite[Theorem 3.5]{Cruz_Uribe_Bownik_Extrapolation}, we immediately deduce the complex version of \cite[Lemma 3.8]{Cruz_Uribe_Bownik_Extrapolation}, since the composition of measurable maps remains measurable.
    \end{remark}

    The statement and proof of \cite[Lemma 3.9]{Cruz_Uribe_Bownik_Extrapolation} for complex convex bodies
    correspond to Lemma~\ref{lemma:ModulusSelectionFunction} in the present work.
    The proof included in Section~\ref{sec:ConvexBodyValuedFunctions} already covers the case of complex convex bodies.
    
    Finally, \cite[Lemma 3.10]{Cruz_Uribe_Bownik_Extrapolation} relates measurable complex ellipsoid valued mappings $G$
    to measurable matrix valued mappings with mutually orthogonal columns.

    \begin{lemma}[Lemma~3.10 in~\cite{Cruz_Uribe_Bownik_Extrapolation}]
        \label{lemma:JohnEllipsoidMatrix}
        Consider a measurable mapping $G\colon \Omega \rightarrow \convex(\complex^d)$
        such that $G(x)$ is a complex ellipsoid for every $x \in \Omega$
        (possibly with empty interior).
        Then there exists a measurable mapping $W\colon \Omega \rightarrow \matrices{d}{\complex}$ such that
        \begin{enumerate}[(i)]
            \item
            the columns of the matrix $W(x)$ are mutually orthogonal,
            \item 
            it holds that $G(x) = W(x) \complexball{d}$ for every $x \in \Omega.$
        \end{enumerate}
    \end{lemma}

    \begin{remark}
        The proof consists in constructing the columns $v_1,\ldots,v_d$ of the matrix $W(x)$
        as measurable mappings $\Omega \rightarrow \complex^d$ and being mutually orthogonal
        (with some of them possibly null at a given $x \in \Omega$).
        This construction is the same for the complex case than for the real case.
        The major difference being only that, instead of using the rationals $\mathbb{Q}$ as a dense set of $\reals,$
        one needs to take the field of Gaussian rationals $\mathbb{Q}(i) = \{p+iq\colon p,q\in\mathbb{Q}\}$
        as a dense set of $\complex.$
    \end{remark}

    Once one has the complex setting versions of~\cite[Lemmata~3.8, 3.9 and~3.10]{Cruz_Uribe_Bownik_Extrapolation},
    the proof of \cite[Theorem 3.7]{Cruz_Uribe_Bownik_Extrapolation} follows equally by replacing $\reals^d$ by $\complex^d,$
    and modifying the related notions accordingly (e.g. the Euclidean scalar product by the hermitian one).

    We make some final remarks about the second part of \cite[Section 3]{Cruz_Uribe_Bownik_Extrapolation}, which concerns integrals of convex-set valued maps.
    The definitions of the Aumann integral and integrable bounded functions there are valid with $\reals^{d}$ replaced by $\reals^{2d}$ and thus also by $\complex^d$.
    Note that the definition of the Aumann integral makes use only of the topological space and real vector space structures of $\reals^d$.
    Note also that complex symmetric sets are in particular real symmetric sets.
    These general principles can be applied through the rest of \cite[Section 3]{Cruz_Uribe_Bownik_Extrapolation} to obtain the complex versions of all results concerning the Aumann integral.

    Finally, another main point in \cite[Section 3]{Cruz_Uribe_Bownik_Extrapolation} is that
    Lemma~3.13 in~\cite{Cruz_Uribe_Bownik_Extrapolation} yields,
    for a given measurable real vector valued function $f,$
    a measurable real convex body valued function $F$ with $|F(x)| = |f(x)|$ at every $x \in \Omega.$
    For the complex version, we consider a measurable complex vector valued function $f$
    and we construct the corresponding measurable complex convex body valued function $F$
    with $|F(x)| = |f(x)|$ at every point in the domain.
    This is achieved performing one major change that is a recurring theme in the passage from the real to the complex case. Namely, given $f\in\mathrm{L}^1(\Omega,\complex^d)$, one defines the convex-set valued map $F$ by
    \begin{equation*}
        F(x):=\{zf(x):~z\in\complex\text{ with }|z|\leq1\},\quad x\in\Omega.
    \end{equation*}
    This map is measurable just as in the real case in \cite[Lemma 3.13]{Cruz_Uribe_Bownik_Extrapolation} because also the closed unit disk in the complex plane has a dense countable subset.

    Finally, the rest of the results in \cite[Section 3]{Cruz_Uribe_Bownik_Extrapolation} follow by the general principles mentioned before.
    Namely, one uses that $\complex^d$ shares an equivalent structure to $\reals^{2d}.$
    In addition, all abstract theorems used to treat Aumann integrals are equally valid in the complex setting as in the real one.

    \subsection{Seminorm functions}

    The adaptation of~\cite[Section~4]{Cruz_Uribe_Bownik_Extrapolation} is performed similarly to that for~\cite[Section~2]{Cruz_Uribe_Bownik_Extrapolation} outlined above, using the corresponding results from the previously described adaptation of~\cite[Section~3]{Cruz_Uribe_Bownik_Extrapolation}. Additional ingredients are countable dense sets of $\complex^d$ such as $\mathbb{Q}(i)^d$
    and the complex version of Hahn-Banach Theorem. Finally, note that the characterization of the convergence of convex bodies used in~\cite[Section~4]{Cruz_Uribe_Bownik_Extrapolation}
    can be used in the same way in our context
    because it actually applies to general nonempty compact convex sets (see~\cite[Theorem~1.8.7]{Schneider_Convex_Bodie_1993}).

    \subsection{Main results including the Extrapolation Theorem with matrix weights}
    In the previous subsections we have explained how the theory developed to deal with real valued matrix weights
    in~\cite[Sections~2--4]{Cruz_Uribe_Bownik_Extrapolation}
    can be modified to be applied to complex valued matrix weights.
    Once these tools have been conveniently adapted,
    they can be used to get the complex version of the main results in that article
    without further changes.
    In other words, the exposition of~\cite[Sections~5--9]{Cruz_Uribe_Bownik_Extrapolation}
    holds for the complex setting by using the results explained in the current appendix.
    In particular, Theorem~\ref{thm:Extrapolation} holds.
    
    \printbibliography

    \Addresses

\end{document}